\documentclass[12pt]{article}
\usepackage{amssymb}
\usepackage{amsmath}
\usepackage{amsthm}
\usepackage{color}
\usepackage{comment}
\usepackage{enumitem}
\usepackage{graphicx}
\begin{document}

\def\cO{\mathcal{O}}
\def\cS{\mathcal{S}}
\def\cX{\mathcal{X}}
\def\cY{\mathcal{Y}}
\def\cSp{\mathcal{S}'}
\def\sJ{J}
\def\sL{L}
\def\sR{R}

\newcommand{\removableFootnote}[1]{}

\newtheorem{theorem}{Theorem}[section]
\newtheorem{lemma}[theorem]{Lemma}
\theoremstyle{remark}
\newtheorem{remark}{Remark}[section]

\title{
Sequences of Periodic Solutions and Infinitely Many Coexisting Attractors in the Border-Collision Normal Form.
}
\author{
D.J.W.~Simpson\\
Institute of Fundamental Sciences\\
Massey University\\
Palmerston North\\
New Zealand}
\maketitle

\begin{abstract}
The border-collision normal form is a piecewise-linear continuous map on $\mathbb{R}^N$
that describes dynamics near border-collision bifurcations of nonsmooth maps.
This paper studies a codimension-three scenario at which the border-collision normal form with $N=2$ exhibits infinitely many attracting periodic solutions.
In this scenario there is a saddle-type periodic solution with branches of stable and unstable manifolds that are coincident,
and an infinite sequence of attracting periodic solutions that converges to an orbit homoclinic to the saddle-type solution.
Several important features of the scenario are shown to be universal, and three examples are given.
For one of these examples infinite coexistence is proved directly by explicitly computing periodic solutions in the infinite sequence.
\end{abstract}


\section{Introduction}
\label{sec:intro}

For a map that is smooth except on codimension-one switching manifolds where it is only continuous,
a border-collision bifurcation occurs when a fixed point of the map collides with a switching manifold under parameter change,
and local to the bifurcation the map is asymptotically piecewise-linear \cite{DiBu08,ZhMo03,DiFe99}.
Except in special cases, dynamical behavior near a border-collision bifurcation is completely determined by the linear components.
Upon omitting higher order terms and introducing convenient local coordinates, in two dimensions the map may be written as
\begin{equation}
\begin{gathered}
\left[ \begin{array}{c} x_{i+1} \\ y_{i+1} \end{array} \right] = 
\left\{ \begin{array}{lc}
f^{\sL}(x_i,y_i) \;, & x_i \le 0 \\
f^{\sR}(x_i,y_i) \;, & x_i \ge 0
\end{array} \right. \;, \\
f^{\sJ}(x,y) = A_{\sJ}
\left[ \begin{array}{c} x \\ y \end{array} \right] +
\left[ \begin{array}{c} 1 \\ 0 \end{array} \right] \mu \;, \qquad
A_{\sJ} = \left[ \begin{array}{cc} \tau_{\sJ} & 1 \\ -\delta_{\sJ} & 0 \end{array} \right] \;, \qquad
\sJ \in \{ \sL, \sR \} \;.
\end{gathered}
\label{eq:f}
\end{equation}
The two-dimensional border-collision normal form (\ref{eq:f}) is piecewise-linear and continuous on the single switching manifold, $x=0$.
The four parameters, $\tau_{\sL}$, $\delta_{\sL}$, $\tau_{\sR}$ and $\delta_{\sR}$, may take any value in $\mathbb{R}$.
The remaining parameter $\mu$ controls the border-collision bifurcation.
The bifurcation occurs at $\mu = 0$, and in the context of border-collision $\mu$ is presumed to be small.
However, for $\mu \ne 0$ the structure of the dynamics of (\ref{eq:f}) is independent of the magnitude of $\mu$
because the half-maps $f^{\sL}$ and $f^{\sR}$ are affine.
All bounded invariant sets of (\ref{eq:f}) collapse to the origin as $\mu \to 0$.
Hence for the purposes of determining the behavior of (\ref{eq:f}) it suffices to assume $\mu \in \{ -1, 0, 1 \}$.
In $N$ dimensions, the matrices in the border-collision normal form are $N \times N$ companion matrices
and there are a total of $2 N$ parameters in addition to $\mu$ \cite{Di03}.

The construction and basic properties of (\ref{eq:f}) were first described by Nusse and Yorke \cite{NuYo92}\removableFootnote{
though some study of piecewise-linear continuous maps had been performed earlier,
refer to references within \cite{DiFe99,MiGa96}
}.
The border-collision normal form also arises as a Poincar\'{e} map for corner-collisions in Filippov systems \cite{DiBu01c},
as well as grazing-sliding bifurcations in these systems,
albeit with the restriction that one of the two matrices in the map has a zero eigenvalue \cite{DiKo02,KuRi03,Co08b}.
Recently it has been shown that dynamics near a grazing-sliding bifurcation in an $(N+2)$-dimensional system
may be partially captured by the $N$-dimensional border-collision normal form
because sliding motion relates to a loss of dimension \cite{GlJe12}.
Various piecewise-linear continuous maps (that may be transformed to the normal form)
have been used as mathematical models, particularly in social sciences \cite{PuSu06}.

Since (\ref{eq:f}) is piecewise-linear, it is extremely nonlinear yet relatively amenable to an exact analysis.
Glendinning and Wong showed that (\ref{eq:f}) may exhibit an attractor that fills a two-dimensional region
of phase space by constructing Markov partitions to calculate this region exactly \cite{GlWo11}.
Arnold tongues of (\ref{eq:f}) typically display a chain structure with points of zero width
at which there exists an invariant polygon \cite{ZhMo06b,SuGa08,SiMe09}. 
The map (\ref{eq:f}) may have a unique fixed point for all $\mu$
that is asymptotically stable for all $\mu \ne 0$, yet unstable when $\mu = 0$ \cite{Do07,HaAb04}.
Additional dynamics is possible when (\ref{eq:f}) is non-invertible \cite{MiGa96},
such as snap-back repellers which imply chaotic dynamics
and the coexistence of infinitely many unstable periodic solutions \cite{Gl10}.

Several authors have described coexisting attractors for (\ref{eq:f}).
Since bounded attractors converge to the origin as $\mu \to 0$,
coexistence produces an unavoidable uncertainty near the border-collision bifurcation in the presence of small noise \cite{DuNu99}.
The coexistence of six attracting periodic solutions for (\ref{eq:f}) was noted briefly in \cite{Si10}.
The purpose of this paper is to show that (\ref{eq:f}) may exhibit infinitely many attracting periodic solutions,
and to describe conditions on the parameter values that indicate when this phenomenon may occur.

To study orbits of (\ref{eq:f}) it is helpful to consider symbol sequences, $\cS : \mathbb{Z} \to \{ \sL, \sR \}$.
We can associate such a symbol sequence to any orbit of (\ref{eq:f}) by setting
$\cS_i = \sL$ if $x_i < 0$ and
$\cS_i = \sR$ if $x_i > 0$, for all $i \in \mathbb{Z}$.
(If $x_i = 0$, it is convenient to place no restriction on $\cS_i$ because (\ref{eq:f}) is a continuous map.)
Conversely, given a symbol sequence $\cS$ and an initial point $(x_0,y_0)$,
we can define a forward orbit that follows $\cS$ by setting $(x_{i+1},y_{i+1}) = f^{\cS_i}(x_i,y_i)$\removableFootnote{
If (\ref{eq:f}) is invertible we may define a full orbit.
}.
In general, iterating the two half-maps of (\ref{eq:f}) in this fashion
produces a different orbit than iterations of (\ref{eq:f}).
However, if the resulting orbit is {\it admissible},
that is $x_i \le 0$ whenever $\cS_i = \sL$, and $x_i \ge 0$ whenever $\cS_i = \sR$,
then the orbit is identical to that produced by iterating (\ref{eq:f}).
An orbit that is not admissible is said to be {\it virtual}.
If $\cS$ is periodic, a periodic solution that follows $\cS$ is referred to as an {\it $\cS$-cycle}.

\begin{figure}[t!]
\begin{center}
\setlength{\unitlength}{1cm}
\begin{picture}(15,7.5)
\put(0,0){\includegraphics[height=7.5cm]{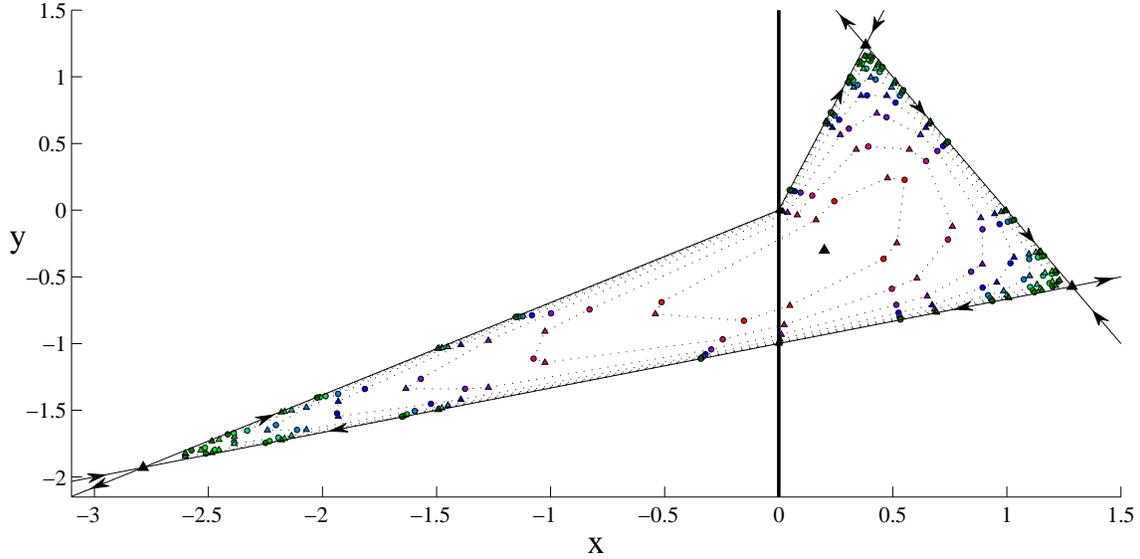}}
\end{picture}
\caption{
A phase portrait of the two-dimensional border-collision normal form (\ref{eq:f}) with $\mu = 1$ and (\ref{eq:paramF}), at which there are infinitely many attractors.
There is a $\sR \sL \sR$-cycle (period-$3$ solution with symbol sequence $\sR \sL \sR$) of saddle-type.
Branches of the stable and unstable manifolds of the $\sR \sL \sR$-cycle (with stability indicated by arrows) are coincident 
and together with the $\sR \sL \sR$-cycle form an invariant quadrilateral.
For each $k \in \mathbb{Z}^+$, there is an attracting $\cS[k]$-cycle, where $\cS[k] = (\sR \sL \sR)^k \sL \sR$, (\ref{eq:SRLRkLR}), 
and a saddle-type $\cS'[k]$-cycle, where $\cS'[k] = (\sR \sL \sR)^k \sR \sR$, (\ref{eq:SRLRkRR}).
The $\cS[k]$ and $\cS'[k]$-cycles are indicated by small circles and triangles, respectively, up to $k = 8$.
So that these periodic solutions may be distinguished clearly,
for each $k$, points of the $\cS[k]$ and $\cS'[k]$-cycles are connected with dotted line segments.
Take care to note that these line segments are not invariants and do not relate to dynamics of the map.
Also shown is the fixed point of $f^{\sR}$, $\left( \frac{1}{5},\frac{-3}{10} \right)$,
which lies in the right half-plane and thus is a fixed point of (\ref{eq:f}).
\label{fig:infFa}
}
\end{center}
\end{figure}

As an example, (\ref{eq:f}) has infinitely many attracting periodic solutions when $\mu = 1$ and the remaining parameter values are given by
\begin{equation}
\tau_{\sL} = -\frac{55}{117} \;, \qquad
\delta_{\sL} = \frac{4}{9} \;, \qquad
\tau_{\sR} = -\frac{5}{2} \;, \qquad
\delta_{\sR} = \frac{3}{2} \;.
\label{eq:paramF}
\end{equation}
Fig.~\ref{fig:infFa} shows a phase portrait.
Attracting periodic solutions are $\cS[k]$-cycles, where $k \in \mathbb{Z}^+$ (the set of positive integers) and
\begin{equation}
\cS[k] = (\sR \sL \sR)^k \sL \sR \;,
\label{eq:SRLRkLR}
\end{equation}
is a periodic symbol sequence of period $3k+2$.
This example exhibits several features that in later sections are shown to be universal.
There is an $\sR \sL \sR$-cycle of saddle-type;
specifically its stability multipliers are $\frac{6}{13}$ and $\frac{13}{6}$.
As $k \to \infty$, the $\cS[k]$-cycles approach an orbit that is homoclinic to the $\sR \sL \sR$-cycle.
In comparison, near homoclinic and heteroclinic orbits of ODEs there may be infinitely many attractors \cite{PaTa93,Ch97}, 
and for area-preserving maps there may be infinitely many elliptic periodic solutions \cite{GoSh00}.
However, in Fig.~\ref{fig:infFa} the intersection of the stable and unstable manifolds of the $\sR \sL \sR$-cycle is non-transversal.
The branches of the stable and unstable manifolds that intersect are coincident and there is no topological horseshoe.
Furthermore, for every $k \in \mathbb{Z}^+$,
there exist saddle-type $\cS'[k]$-cycles, where
\begin{equation}
\cS'[k] = (\sR \sL \sR)^k \sR \sR \;.
\label{eq:SRLRkRR}
\end{equation}
Each $\cS'[k]$ differs from $\cS[k]$ by a single symbol.
The stable manifolds of the $\cS'[k]$-cycles appear to form the boundaries of the basins of attraction of the $\cS[k]$-cycles
that are shown in Fig.~\ref{fig:infFb}.

The remainder of this paper is organized as follows.
Conditions for the existence, admissibility and stability of periodic solutions to (\ref{eq:f}) are given in \S\ref{sec:periodic}.
Periodic solutions may be found by solving linear matrix equations because any composition of $f^{\sL}$ and $f^{\sR}$ is affine.
Sequences of symbol sequences of the form $\cS[k] = \cX^k \cY$ are considered in \S\ref{sec:necessary}.
The main result of this section is Theorem \ref{th:codim3} that
gives consequences of the existence of infinitely many stable $\cS[k]$-cycles
when the $\cX$-cycle is of saddle-type (as in Fig.~\ref{fig:infFa} for which $\cX = \sR \sL \sR$).
The theorem reveals three necessary conditions on the parameter values, from which we find that this scenario is codimension-three,
and that the $\cS[k]$-cycles limit to a homoclinic orbit as $k \to \infty$.
In \S\ref{sec:further}, additional assumptions are placed on $\cX$ and $\cY$
leading to further consequences, such as the coincidence of branches of the stable and unstable manifolds of the $\cX$-cycle.
In \S\ref{sec:finding}, the results are used to obtain parameter values of infinite coexistence
for three different choices of $\cX$ and $\cY$.
For one of these choices, corresponding to Figs.~\ref{fig:infFa} and \ref{fig:infFb},
the existence of attracting $\cS[k]$-cycles and saddle-type $\cS'[k]$-cycles for all $k \in \mathbb{Z}^+$
is demonstrated formally in \S\ref{sec:sufficient}.
Conclusions and future directions are discussed in \S\ref{sec:conc}.

\begin{figure}[t!]
\begin{center}
\setlength{\unitlength}{1cm}
\begin{picture}(15,7.5)
\put(0,0){\includegraphics[height=7.5cm]{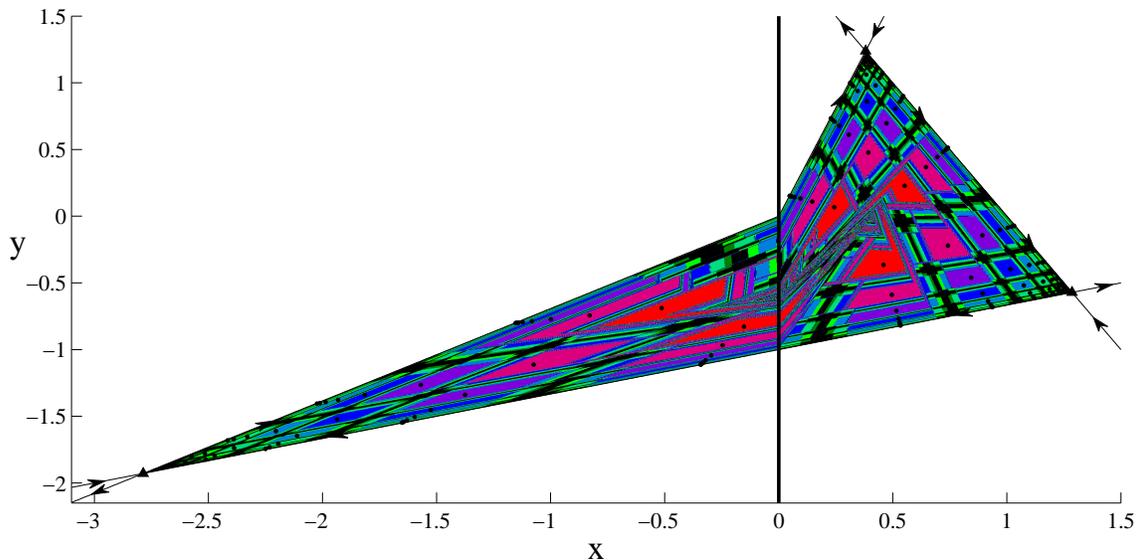}}
\end{picture}
\caption{
Basins of attraction of the $\cS[k]$-cycles of Fig.~\ref{fig:infFa} up to $k = 8$
computed numerically by iterating (\ref{eq:f}) from a $2048 \times 1536$ grid of initial points.
The color of each of the eight basins matches the color of the $\cS[k]$-cycle shown in Fig.~\ref{fig:infFa}, with red for $k=1$ and dark green for $k=8$.
Initial points are shaded white if the forward orbit appeared to diverge,
and shaded black if the forward orbit appeared to neither converge to an $\cS[k]$-cycle with $k \le 8$, or diverge.
For clarity, the $\cS'[k]$-cycles are not shown.
\label{fig:infFb}
}
\end{center}
\end{figure}

\section{Periodic solutions}
\label{sec:periodic}

Calculations of periodic solutions of $N$-dimensional piecewise-linear continuous maps
are given in \cite{SiMe09,Si10,SiMe10}.
In this section these calculations are summarized for the two-dimensional normal form (\ref{eq:f}).



Let $\cS : \mathbb{Z} \to \{ \sL, \sR \}$ be a periodic symbol sequence with minimal period $n \ge 1$
(that is, $\cS_{i+n} = \cS_i$, for all $i \in \mathbb{Z}$,
and $\cS$ does not exhibit this property for a smaller value of $n$).
Then the word $\cS_0 \cdots \cS_{n-1}$ is {\it primitive}, that is, cannot be written as a power
(e.g.~$\sR \sL \sR$ is primitive, but $\sR \sL \sR \sL = (\sR \sL)^2$ is not).
Conversely, given a primitive word of length $n$,
the sequence defined by the infinite repetition of this word has minimal period $n$.
For this reason, throughout this paper whenever we write a periodic symbol sequence,
as in (\ref{eq:SRLRkLR}) and (\ref{eq:SRLRkRR}), we list the symbols $\cS_0 \cdots \cS_{n-1}$\removableFootnote{
Note:
(i) sequences are infinite, e.g.~$\cS = \cdots \sR \sL \sR \sR \sL \sR \cdots$,
(ii) but for brevity we write periodic symbol sequences as, e.g.~$\cS = \sR \sL \sR$.
This is misleading because (iii) a finite list of symbols is a word.
The notation is okay because (iv) periodic symbol sequences of minimal period $n$ are isomorphic(?) to primitive words of length $n$.
}.

Let
\begin{equation}
f^{\cS} = f^{\cS_{n-1}} \circ \cdots \circ f^{\cS_0} \;,
\label{eq:fS}
\end{equation}
denote the $n^{\rm th}$ iterate of (\ref{eq:f}) following $\cS$.
A straight-forward expansion leads to
\begin{equation}
f^{\cS}(x,y) = M_{\cS} \left[ \begin{array}{c} x \\ y \end{array} \right] +
P_{\cS} \left[ \begin{array}{c} 1 \\ 0 \end{array} \right] \mu \;,
\label{eq:fS2}
\end{equation}
where
\begin{equation}
M_{\cS} = A_{\cS_{n-1}} \cdots A_{\cS_0} \;, \qquad
P_{\cS} = I + A_{\cS_{n-1}} + A_{\cS_{n-1}} A_{\cS_{n-2}} +
\cdots + A_{\cS_{n-1}} \cdots A_{\cS_1} \;.
\label{eq:MSPS}
\end{equation}
Let $\cS^{(i)}$ denote the $i^{\rm th}$ left shift permutation of $\cS$
(e.g.~if $\cS = \sL \sL \sL \sR \sR$, then $\cS^{(2)} = \sL \sR \sR \sL \sL$).
The $i^{\rm th}$ point of an $\cS$-cycle, denoted $\left( x^{\cS}_i,y^{\cS}_i \right)$,
is a fixed point of $f^{\cS^{(i)}}$.
When $I - M_{\cS^{(i)}}$ is non-singular, this point is unique.
Since the spectrum of $I - M_{\cS^{(i)}}$ is independent of $i$\removableFootnote{
Let $\lambda$ be an eigenvalue of $I - M_{\cS}$.\\
Then $(I - M_{\cS}) v = \lambda v$, for some nonzero vector $v \in \mathbb{R}^2$.\\
By multiplying both sides of this equation by $A_{\cS_0}$ on the left, we obtain
$\left( I - M_{\cS^{(1)}} \right) A_{\cS_0} v = \lambda A_{\cS_0} v$.\\
Therefore $\lambda$ is an eigenvalue of $I - M_{\cS^{(1)}}$.\\
By either a repetition or generalization of this argument we conclude
that $\lambda$ is an eigenvalue of $I - M_{\cS^{(i)}}$, for any $i$.
},
we have the following result.

\begin{lemma}[Existence]
The $\cS$-cycle is unique if and only if $\det \left( I - M_{\cS} \right) \ne 0$.
Moreover, if $\det \left( I - M_{\cS} \right) \ne 0$, then for each $i$,
\begin{equation}
\left[ \begin{array}{c} x^{\cS}_i \\ y^{\cS}_i \end{array} \right] =
\left( I - M_{\cS^{(i)}} \right)^{-1} P_{\cS^{(i)}} \left[ \begin{array}{c} 1 \\ 0 \end{array} \right] \mu \;.
\label{eq:existence}
\end{equation}
\label{le:existence}
\end{lemma}

An $\cS$-cycle is admissible if every point lies on the ``correct'' side of $x=0$, or on $x=0$.
The following formula results from manipulating (\ref{eq:existence}) (see \cite{Si10,SiMe10})\removableFootnote{
Here $\varrho = (1,1)^{\sf T}$ and $b = (1,0)^{\sf T}$,
thus $\varrho^{\sf T} b = 1$.
}:
\begin{equation}
x^{\cS}_i = \frac{\det \left( P_{\cS^{(i)}} \right) \mu}{\det(I - M_{\cS})} \;.
\nonumber
\end{equation}
Admissibility is therefore determined by the signs of $\det \left( P_{\cS^{(i)}} \right)$,
as described in the following lemma\removableFootnote{
Furthermore, if the $\cS$-cycle is admissible,
then the $n$ points are distinct because $n$ is the minimal period.
}.

\begin{lemma}[Admissibility]
Suppose $\mu \ne 0$ and $\det(I - M_{\cS}) \ne 0$.
Then the $\cS$-cycle is an admissible periodic solution of (\ref{eq:f}) if and only if, whenever
$\det \left( P_{\cS^{(i)}} \right) \ne 0$,
\begin{equation}
\begin{gathered}
{\rm if~} \cS_i = {\sL}, {\rm ~then~}
{\rm sgn} \left( \det \left( P_{\cS^{(i)}} \right) \right) = -{\rm sgn} \left( \mu \det \left( I - M_{\cS} \right) \right) \;, \\
{\rm and~if~} \cS_i = {\sR}, {\rm ~then~}
{\rm sgn} \left( \det \left( P_{\cS^{(i)}} \right) \right) = {\rm sgn} \left( \mu \det \left( I - M_{\cS} \right) \right) \;.
\nonumber
\end{gathered}
\end{equation}
\label{le:admissibility}
\end{lemma}

If no points of an admissible $\cS$-cycle lie on the switching manifold 
then there exists a neighborhood of each point $\left( x^{\cS}_i,y^{\cS}_i \right)$
for which the $n^{\rm th}$ iterate of (\ref{eq:f}) is given by $f^{\cS^{(i)}}$.
In this case the $\cS$-cycle is asymptotically stable
if and only if both eigenvalues of $M_{\cS}$ (these are the stability multipliers of the $\cS$-cycle)
lie inside the unit circle.
For two-dimensional maps it is well-known that stability relates to a particular triangle in the
space of coordinates ${\rm trace} \left( M_{\cS} \right)$
and $\det \left( M_{\cS} \right)$ \cite{El08,Ga07,MeLi01},
and we have the following result.

\begin{lemma}[Stability]
Suppose $\mu \ne 0$, $\det \left( I - M_{\cS} \right) \ne 0$, $\det \left( P_{\cS^{(i)}} \right) \ne 0$ for all $i$,
and the $\cS$-cycle is admissible.
Then the $\cS$-cycle is an asymptotically stable periodic solution of (\ref{eq:f}) if and only if
the following three conditions are satisfied
\begin{align}
\det(M_{\cS}) - {\rm trace}(M_{\cS}) + 1 &> 0 \;, \label{eq:stabConditionSN} \\
\det(M_{\cS}) + {\rm trace}(M_{\cS}) + 1 &> 0 \;, \label{eq:stabConditionPD} \\
\det(M_{\cS}) - 1 &< 0 \;. \label{eq:stabConditionNS}
\end{align}
\label{le:stability}
\end{lemma}
If there is equality in at least one of (\ref{eq:stabConditionSN})-(\ref{eq:stabConditionNS})
but the conditions are satisfied otherwise, then the $\cS$-cycle is stable but not asymptotically stable. 
Equality in (\ref{eq:stabConditionSN})-(\ref{eq:stabConditionNS}) corresponds to, in order,
an eigenvalue $1$,
an eigenvalue $-1$,
and a complex pair of eigenvalues on the unit circle when $|{\rm trace}(M_S)| < 2$.
Note that $\det \left( I - M_{\cS} \right) \equiv \det(M_{\cS}) - {\rm trace}(M_{\cS}) + 1$,
hence the assumption $\det \left( I - M_{\cS} \right) \ne 0$ eliminates the possibility of equality in (\ref{eq:stabConditionSN}).

\section{Necessary conditions for infinite coexistence} 
\label{sec:necessary}

In this section we consider sequences of symbol sequences of the form $\cS[k] = \cX^k \cY$
and obtain conditions on the parameter values of (\ref{eq:f}) that are necessary in order
for infinitely many $\cS[k]$-cycles to be admissible and stable.
The main result is Theorem \ref{th:codim3}.
First some additional notation is introduced.

Suppose that for a given periodic symbol sequence $\cX$,
the matrix $M_{\cX}$ has distinct real eigenvalues $\lambda_1$ and $\lambda_2$, neither of which are equal to $1$.
In this case $\det \left( I - M_{\cX} \right) \ne 0$,
so by Lemma \ref{le:existence} the $\cX$-cycle is unique.
We write the $\cX$-cycle as $\left( x^{\cX}_i,y^{\cX}_i \right)$, for $i=0,\ldots,n_{\cX}-1$,
where $n_{\cX}$ denotes the minimal period of $\cX$.
From (\ref{eq:existence}), with $i=0$
\begin{equation}
\left[ \begin{array}{c} x^{\cX}_0 \\ y^{\cX}_0 \end{array} \right] =
\left( I - M_{\cX} \right)^{-1} P_{\cX} \left[ \begin{array}{c} 1 \\ 0 \end{array} \right] \mu \;.
\label{eq:zStarX}
\end{equation}
Let $\zeta_1$ and $\zeta_2$ be eigenvectors of $M_{\cX}$, corresponding to $\lambda_1$ and $\lambda_2$ respectively,
and let $Q = \big[ \zeta_1 \; \zeta_2 \big]$.
We then consider the change of coordinates
\begin{equation}
\left[ \begin{array}{c} u \\ v \end{array} \right] =
Q^{-1} \left( \left[ \begin{array}{c} x \\ y \end{array} \right] -
\left[ \begin{array}{c} x^{\cX}_0 \\ y^{\cX}_0 \end{array} \right] \right) \;,
\label{eq:uv}
\end{equation}
and, for any $\cS$, let $g^{\cS}$ denote $f^{\cS}$ in $(u,v)$-coordinates.
The coordinates (\ref{eq:uv}) are defined such that $g^{\cX}$ is linear and completely decoupled, specifically
\begin{equation}
g^{\cX}(w) = \left[ \begin{array}{cc} \lambda_1 & 0 \\ 0 & \lambda_2 \end{array} \right] w \;,
\label{eq:gX}
\end{equation}
where we let $w = (u,v)$.

\begin{theorem}
Let $\cS[k] = \cX^k \cY$, where $\cX$ is primitive and $\cX_0 \ne \cY_0$.
Let $\tau_{\sL}, \delta_{\sL}, \tau_{\sR}, \delta_{\sR} \in \mathbb{R}$ and $\mu \ne 0$.
Suppose there exist infinitely many values of $k \ge 1$ for which the map (\ref{eq:f})
exhibits a unique, admissible, stable $\cS[k]$-cycle that has no points on the switching manifold.
Suppose $\lambda_1$ and $\lambda_2$ are eigenvalues of $M_{\cX}$, where $0 \le \lambda_1 < 1 < \lambda_2$,
and that the $v$-axis (as defined by (\ref{eq:uv})) is not parallel to the switching manifold ($x=0$).
Then,
\begin{enumerate}[label=\roman{*}),ref=\roman{*}]
\item
\label{it:gY}
$g^{\cY}$ maps the $v$-axis to the $u$-axis;
\item
\label{it:lambda2}
$\lambda_1 \ne 0$ and $\lambda_2 = \frac{1}{\lambda_1}$;
\item
\label{it:HCorbit}
the $\cX$-cycle is admissible
and $\cS[k]$-cycles limit to an orbit that is homoclinic to the $\cX$-cycle as $k \to \infty$.
\end{enumerate}
\label{th:codim3}
\end{theorem}

A proof of Theorem \ref{th:codim3} is given after some remarks.
First notice that the scenario depicted in Fig.~\ref{fig:infFa} conforms to the assumptions of Theorem \ref{th:codim3}.
Here $\cX = \sR \sL \sR$, which is primitive,
and $\cY = \sL \sR$, which begins with a different symbol than $\cX$.
The $(u,v)$-coordinates are centered at $\left( x^{\sR \sL \sR},y^{\sR \sL \sR} \right)$ -- the right-most point of the $\sR \sL \sR$-cycle.
Locally, the $u$ and $v$-axes are, respectively, the stable and unstable manifolds of this point,
and it is straight-forward to verify directly that parts (\ref{it:gY})-(\ref{it:HCorbit}) of Theorem \ref{th:codim3} are satisfied for this example.

The form $\cS[k] = \cX^k \cY$, with the assumptions of Theorem \ref{th:codim3}, is highly general.
Given $\cS[k] = \cX^k \cY$ with $\cX$ not primitive,
we may redefine $\cX$ so that it is primitive by introducing a higher power of $k$.
Also if $\cX_0 = \cY_0$, or $\cS[k]$ involves symbols preceding $\cX^k$,
as long as $\cY$ is not a power of $\cX$ we may apply a shift permutation and redefine $\cX$ and $\cY$
so that $\cS[k]$ takes the form $\cX^k \cY$ with $\cX_0 \ne \cY_0$.
Also, the assumption that $\cS[k]$-cycles have no points on the switching manifold
is made in part for simplicity -- so that their stability is determined purely by the eigenvalues of $M_{\cS[k]}$
as opposed to more sophisticated methods \cite{DoKi08} --
and in part because the presence of points on the switching manifold represents an additional degeneracy.

The restrictions on the eigenvalues of $M_{\cX}$ are motivated by observations of
infinitely many stable or asymptotically stable periodic solutions in smooth maps
occurring when there is a homoclinic or heteroclinic connection (which requires the existence of an invariant of saddle-type).
It is not clear if infinite coexistence is possible when $M_{\cX}$ has a negative eigenvalue,
because in this instance both branches of the corresponding invariant manifold are involved.

In $(u,v)$-coordinates, $g^{\cY}$ is an affine map;
let us write it as
\begin{equation}
g^{\cY}(w) = \left[ \begin{array}{cc} \gamma_{11} & \gamma_{12} \\ \gamma_{21} & \gamma_{22} \end{array} \right] w +
\left[ \begin{array}{c} \sigma_1 \\ \sigma_2 \end{array} \right] \;,
\label{eq:gY}
\end{equation}
for some constants $\gamma_{ij}$, $\sigma_1$ and $\sigma_2$.
Then part (\ref{it:gY}) of Theorem \ref{th:codim3} is equivalent to the statement, $\gamma_{22} = \sigma_2 = 0$.
Assuming that the three requirements $\lambda_2 = \frac{1}{\lambda_1}$ (part (\ref{it:lambda2}) of the theorem),
$\gamma_{22} = 0$ and $\sigma_2 = 0$ involve no degeneracy,
it follows that the assumptions of Theorem \ref{th:codim3} describe a scenario that is at least codimension-three\removableFootnote{
Careful!
Note, for the three examples below it is evident that the three requirements involve no degeneracy.
On the other hand note that the theorem doesn't tell us if whether or not there may be more requirements.
Given that the three requirements do not involve a degeneracy (except in a higher codimension scenario!)
the three requirements describe a codimension-three scenario
and the theorem tells us that the assumptions of the theorem describe a scenario that is at least codimension-three.
}.

The demonstration of $\sigma_2 = 0$ in the proof below
uses the assumption that the $v$-axis is not parallel to the switching manifold
(equivalently, $\left[ 0,1 \right]^{\sf T}$ is not an eigenvector of $M_{\cX}$ for the eigenvalue $\lambda_2$).
As evident in the following proof, with $\sigma_2 \ne 0$,
$\cS[k]$-cycles grow in size with $k$ without bound and are virtual for large $k$
if the $v$-axis is not parallel to the switching manifold.
It remains to determine if infinite coexistence is achievable for (\ref{eq:f})
in the case that $\sigma_2 \ne 0$ and the $v$-axis is parallel to the switching manifold\removableFootnote{
I did a brief search for parameter values without success.
}.

\begin{proof}[Proof of Theorem \ref{th:codim3}]
Here parts (\ref{it:gY})-(\ref{it:HCorbit}) of Theorem \ref{th:codim3} are demonstrated sequentially.
\begin{enumerate}[label=\roman{*}),ref=\roman{*}]
\item
By composing $k$ instances of (\ref{eq:gX}) with (\ref{eq:gY}),
in $(u,v)$-coordinates the image of a point $w$ under $\cS[k]$ is
\begin{equation}
g^{\cS[k]}(w) =
\left[ \begin{array}{cc}
\gamma_{11} \lambda_1^k & \gamma_{12} \lambda_2^k \\
\gamma_{21} \lambda_1^k & \gamma_{22} \lambda_2^k
\end{array} \right] w +
\left[ \begin{array}{c} \sigma_1 \\ \sigma_2 \end{array} \right] \;.
\label{eq:gSk}
\end{equation}
The matrix part of (\ref{eq:gSk}) has the same spectrum as $M_{\cS[k]}$
(the matrix part of $f^{\cS[k]}$), therefore
\begin{equation}
{\rm trace} \left( M_{\cS[k]} \right) = \gamma_{11} \lambda_1^k + \gamma_{22} \lambda_2^k \;.
\nonumber
\end{equation}
By Lemma \ref{le:stability}, the assumption that $\cS[k]$-cycles are stable for large $k$ implies
${\rm trace} \left( M_{\cS[k]} \right) \not\to \infty$ as $k \to \infty$.
Therefore we must have $\gamma_{22} = 0$, since $\lambda_2 > 1$.

With $\gamma_{22} = 0$, ${\rm trace} \left( M_{\cS[k]} \right) \to 0$, as $k \to \infty$, and\removableFootnote{
For stability we require $\det \left( M_{\cS[k]} \right) \not\to \infty$,
but this does not (quite) imply $\lambda_1 \lambda_2 \le 1$ because may it be possible to have $\gamma_{12} \gamma_{21} = 0$.
Note that the only way by which we may have $\lambda_1 \lambda_2 > 1$
and $\gamma_{12} \gamma_{21} = 0$ is if $\cX$ consists of a single symbol.
}
\begin{equation}
\det \left( M_{\cS[k]} \right) = -\gamma_{12} \gamma_{21} \lambda_1^k \lambda_2^k \;.
\nonumber
\end{equation}
We now show that $\det \left( I - M_{\cS[k]} \right)$
is bounded away from zero for large $k$.
Since $\det \left( I - M_{\cS[k]} \right) = \det \left( M_{\cS[k]} \right) - {\rm trace} \left( M_{\cS[k]} \right) + 1$,
this is certainly true if $\det \left( M_{\cS[k]} \right) \to 0$.
If $\det \left( M_{\cS[k]} \right) \not\to 0$,
then, due to the assumption that $\cS[k]$-cycles are stable,
we must have $\lambda_1 \ne 0$, $\lambda_2 = \frac{1}{\lambda_1}$,
and consequently $\det \left( M_{\cS[k]} \right) = -\gamma_{12} \gamma_{21}$ for all $k$.
Furthermore, in this case we cannot have $\det \left( M_{\cS[k]} \right) = -1$ because 
$\cS[k]$-cycles are assumed to be unique and stable for large $k$.
Therefore, in either case, $\det \left( I - M_{\cS[k]} \right)$ is bounded away from zero for large $k$.

In $(u,v)$-coordinates, we denote points of the $\cS[k]$-cycle by
$w^{\cS[k]}_i = \left( u^{\cS[k]}_i,v^{\cS[k]}_i \right)$, for $i = 0,\ldots,k n_{\cX} + n_{\cY} - 1$,
where $n_{\cX}$ and $n_{\cY}$ are the lengths of the words $\cX$ and $\cY$, respectively.
The point, $w^{\cS[k]}_0$, is the unique fixed point of (\ref{eq:gSk}),
and for each $j = 1,\ldots,k$, $w^{\cS[k]}_{j n_{\cX}} = g^{\cX} \left( w^{\cS[k]}_{(j-1) n_{\cX}} \right)$.
From these equations, and substituting $\gamma_{22} = 0$, we obtain
\begin{equation}
w^{\cS[k]}_{j n_{\cX}} = \frac{1}{1 - \gamma_{11} \lambda_1^k - \gamma_{12} \gamma_{21} \lambda_1^k \lambda_2^k}
\left[ \begin{array}{c}
\sigma_1 \lambda_1^j + \gamma_{12} \sigma_2 \lambda_1^j \lambda_2^k \\
\left( \gamma_{21} \sigma_1 - \gamma_{11} \sigma_2 \right) \lambda_1^k \lambda_2^j + \sigma_2 \lambda_2^j
\end{array} \right] \;,
\label{eq:wSkj}
\end{equation}
valid for $j = 0,\ldots,k$.

The symbols $\cS[k]_{(k-1) n_{\cX}} = \cX_0$ and $\cS[k]_{k n_{\cX}} = \cY_0$ are different, by assumption.
Hence, for each $k$ for which the $\cS[k]$-cycle is admissible,
each point $w^{\cS[k]}_{(k-1) n_{\cX}}$ lies on one side of the switching manifold (or on the switching manifold)
and each point $w^{\cS[k]}_{k n_{\cX}}$ lies on the other side of the switching manifold (or on the switching manifold).
We now show that this observation implies
$\sigma_2 = 0$, $\sigma_1 \ne 0$, $\gamma_{21} \ne 0$, $\lambda_1 \ne 0$ and $\lambda_2 = \frac{1}{\lambda_1}$.

By (\ref{eq:wSkj}), if $\sigma_2 \ne 0$, then, as $k \to \infty$,
the $u$ component of $w^{\cS[k]}_{k n_{\cX}}$ converges whereas the $v$ component diverges.
Consequently, any line that divides the points $w^{\cS[k]}_{(k-1) n_{\cX}}$ and $w^{\cS[k]}_{k n_{\cX}}$
for infinitely many values of $k$ must be parallel to the $v$-axis.
The $v$-axis is assumed to be not parallel to the switching manifold, thus this scenario is not permitted.
Therefore $\sigma_2 = 0$ and (\ref{eq:gY}) is given by
\begin{equation}
g^{\cY}(w) = \left[ \begin{array}{cc} \gamma_{11} & \gamma_{12} \\ \gamma_{21} & 0 \end{array} \right] w +
\left[ \begin{array}{c} \sigma_1 \\ 0 \end{array} \right] \;,
\label{eq:gY2}
\end{equation}
which verifies part (\ref{it:gY}).

\item
With $\sigma_2 = 0$, (\ref{eq:wSkj}) reduces to
\begin{equation}
w^{\cS[k]}_{j n_{\cX}} = \frac{\sigma_1}{1 - \gamma_{11} \lambda_1^k - \gamma_{12} \gamma_{21} \lambda_1^k \lambda_2^k}
\left[ \begin{array}{c}
\lambda_1^j \\
\gamma_{21} \lambda_1^k \lambda_2^j
\end{array} \right] \;.
\label{eq:wSkj2}
\end{equation}
Hence $\sigma_1 \ne 0$, since the points $w^{\cS[k]}_{j n_{\cX}}$ must be distinct.
Therefore, as $k \to \infty$, $u^{\cS[k]}_{k n_{\cX}}$ (the $u$-component of (\ref{eq:wSkj2}) with $j = k$) tends to zero.
We cannot have $\gamma_{21} \ne 0$ and $\lambda_1 \lambda_2 > 1$,
because then $\left| v^{\cS[k]}_{k n_{\cX}} \right| \to \infty$ as $k \to \infty$,
which is immediately seen to be not possible in view of the assumption that
the $v$-axis is not parallel to the switching manifold\removableFootnote{
Unlike above, here $u^{\cS[k]}_{k n_{\cX}} \to 0$,
so even with the $v$-axis parallel to the switching manifold
a line cannot be fit between infinitely many $w^{\cS[k]}_{(k-1) n_{\cX}}$ and $w^{\cS[k]}_{k n_{\cX}}$,
although my proof of this is rather long.
}.

If $\gamma_{21} = 0$ or $\lambda_1 \lambda_2 < 1$, then $v^{\cS[k]}_{k n_{\cX}} \to 0$,
which we now show is also not possible.
When the $\cS[k]$-cycle is admissible, under $n_{\cX} + n_{\cY}$ iterations of (\ref{eq:f})
$w^{\cS[k]}_{(k-1) n_{\cX}}$ maps to $w^{\cS[k]}_0$ (following $\cX \cY$)
and $w^{\cS[k]}_{k n_{\cX}}$ maps to $w^{\cS[k]}_{n_{\cX}}$ (following $\cY \cX$).
In this scenario, the distance between the points $w^{\cS[k]}_{(k-1) n_{\cX}}$ and $w^{\cS[k]}_{k n_{\cX}}$ tends to zero as $k \to \infty$,
therefore, since (\ref{eq:f}) is a continuous map,
the distance between the $(n_{\cX} + n_{\cY})^{\rm th}$ iterates of these two points must also tend to zero.
However, from (\ref{eq:wSkj2}) we see that the distance between $w^{\cS[k]}_0$ and $w^{\cS[k]}_{n_{\cX}}$
does not tend to zero, which is contradiction.
Therefore $\gamma_{21} \ne 0$, $\lambda_1 \ne 0$ and $\lambda_2 = \frac{1}{\lambda_1}$.

\item
We now show that the $\cX$-cycle is admissible.
When the $\cS[k]$-cycle is admissible, each $w^{\cS[k]}_{j n_{\cX}}$ for $j = 0,\ldots,k-1$
follows $\cX$ for the next $n_{\cX}$ iterations under (\ref{eq:f}).
By (\ref{eq:wSkj2}), with $j \approx \frac{k}{2}$, $w^{\cS[k]}_{j n_{\cX}} \to (0,0) = w^{\cX}_0$ as $k \to \infty$.
Since (\ref{eq:f}) is continuous and $\cS[k]$-cycles are admissible for infinitely many values of $k$,
the forward orbit of $w^{\cX}_0$ also follows $\cX$ under (\ref{eq:f}), thus the $\cX$-cycle is admissible\removableFootnote{
Conceivably one or more points of the $\cX$-cycle (but not $w^{\cX}_0$) could lie on the switching manifold,
although I have not seen an example of this.
}.

Finally, by (\ref{eq:wSkj2}), as $k \to \infty$, $w^{\cS[k]}_{k n_{\cX}}$
and $w^{\cS[k]}_0$ limit to points on the $v$-axis and $u$-axis respectively.
By the continuity of (\ref{eq:f}), the first limit point belongs to the unstable manifold of $w^{\cX}_0$.
This limit point maps to the second limit point under $n_{\cY}$ iterations of (\ref{eq:f}) (following $\cY$)
and the second limit point belongs to the stable manifold of $w^{\cX}_0$.
Hence these points belong to a homoclinic orbit of the $\cX$-cycle
and $\cS[k]$-cycles limit to this orbit as $k \to \infty$. 
\end{enumerate}
\end{proof}

\section{Further consequences of infinite coexistence}
\label{sec:further}

Part (\ref{it:HCorbit}) of Theorem \ref{th:codim3} tells us that $\cS[k]$-cycles
(comprised of points that in $(u,v)$-coordinates (\ref{eq:uv}) are denoted $w^{\cS[k]}_i$, for $i = 0,\ldots,k n_{\cX} + n_{\cY} - 1$)
limit to an orbit that is homoclinic to the $\cX$-cycle as $k \to \infty$\removableFootnote{
Generically we expect the homoclinic orbit of part (\ref{it:HCorbit}) of Theorem \ref{th:codim3} to have no points on the switching manifold.
In this case, since $\cS[k]$-cycles converge to this orbit
it is immediately evident that not only are there infinitely many values of $k$ for which the $\cS[k]$-cycle
is unique, admissible, and asymptotically stable,
but that this is true for all $k \ge \kappa$, for some $\kappa \in \mathbb{Z}$.
I have omitted this from the statement of Theorem \ref{th:codim3} because it is not terribly important, and as to not crowd the theorem.
If the homoclinic orbit has one or more points on the switching manifold I suspect that this statement still holds
but that a demonstration of the statement is no longer trivial.
Note, the assumption that the homoclinic orbit has no points on the switching manifold
implies $\varphi_0$ and $\psi_0$ lie in the interior of $\Phi_0$.
}.
Therefore the stable and unstable manifolds of the $\cX$-cycle intersect.
In this section it is shown that with additional assumptions that do not
add to the codimension\removableFootnote{
Intuitively this is clear.
Really I can only state this based on the examples in \S\ref{sec:finding} and the proof in \S\ref{sec:sufficient}.
}
of the scenario described by Theorem \ref{th:codim3},
the branches of the stable and unstable manifolds that intersect are coincident.

Before we begin, it is helpful to introduce abbreviated labels to four points of the homoclinic orbit:
\begin{equation}
\begin{split}
a &= \lim_{k \to \infty} w^{\cS[k]}_{(k-1) n_{\cX}} =
\left( 0 \;, \frac{\gamma_{21} \sigma_1 \lambda_1}{1 - \gamma_{12} \gamma_{21}} \right) \;, \qquad
b = \lim_{k \to \infty} w^{\cS[k]}_{k n_{\cX}} =
\left( 0 \;, \frac{\gamma_{21} \sigma_1}{1 - \gamma_{12} \gamma_{21}} \right) \;, \\
c &= \lim_{k \to \infty} w^{\cS[k]}_0 =
\left( \frac{\sigma_1}{1 - \gamma_{12} \gamma_{21}} \;, 0 \right) \;, \qquad
d = \lim_{k \to \infty} w^{\cS[k]}_{n_{\cX}} =
\left( \frac{\sigma_1 \lambda_1}{1 - \gamma_{12} \gamma_{21}} \;, 0 \right) \;,
\end{split}
\label{eq:abcd}
\end{equation}
where the formulas in (\ref{eq:abcd}) are obtained from (\ref{eq:wSkj2}).
Under (\ref{eq:f}), $a$ maps to $b$ following $\cX$,
$b$ maps to $c$ following $\cY$,
and $c$ maps to $d$ following $\cX$.
These points are shown in a schematic of $(u,v)$-coordinates, Fig.~\ref{fig:uvCoords}.
Let $\Phi_0$ denote the closed line segment connecting $a$ and $b$,
and let $\Phi_i$ denote the image of $\Phi_{i-1}$ under (\ref{eq:f}), for each $i = 1,\ldots,n_{\cX}+n_{\cY}$.
Also let,
\begin{equation}
\Xi = \bigcup_{i=0}^{n_{\cX}+n_{\cY}-1} \Phi_i \;.
\nonumber
\end{equation}

\begin{figure}[b!]
\begin{center}
\setlength{\unitlength}{1cm}
\begin{picture}(15,7.5)
\put(0,0){\includegraphics[height=7.5cm]{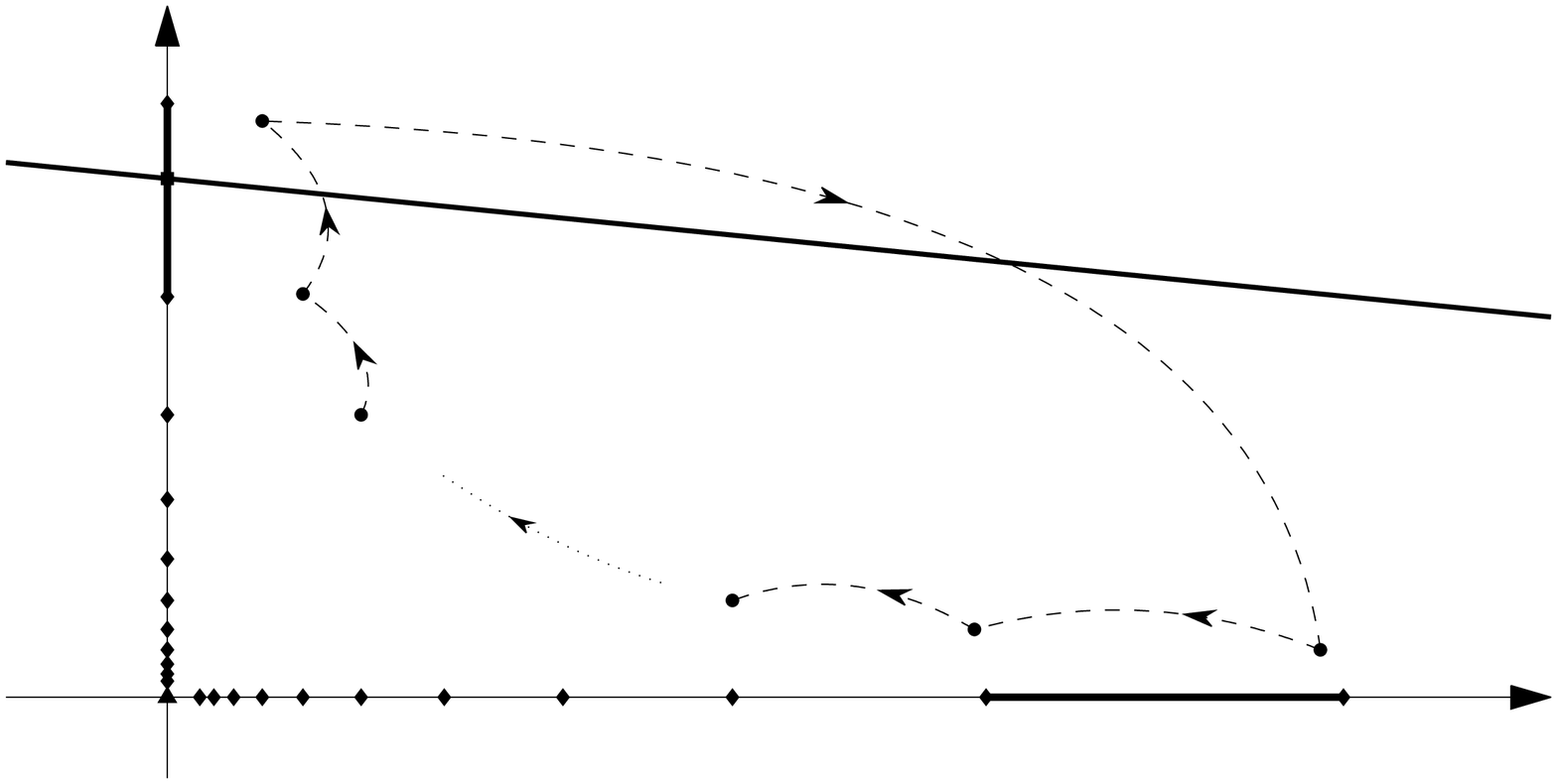}}
\put(13,4.8){\footnotesize $x=0$}
\put(1.1,5.1){\footnotesize $\Phi_0$}
\put(10.9,.46){\footnotesize $\Phi_{n_{\cX}+n_{\cY}}$}
\put(1.63,5.54){\footnotesize $\varphi_0$}
\put(1.67,4.6){\footnotesize $a$}
\put(1.67,6.45){\footnotesize $b$}
\put(12.8,.94){\footnotesize $c$}
\put(9.5,.94){\footnotesize $d$}
\put(3,4.65){\footnotesize $w^{\cS[k]}_{(k-1) n_{\cX}}$}
\put(2.5,6.55){\footnotesize $w^{\cS[k]}_{k n_{\cX}}$}
\put(12.8,1.3){\footnotesize $w^{\cS[k]}_0$}
\put(9.2,1.66){\footnotesize $w^{\cS[k]}_{n_{\cX}}$}
\put(3.6,3.9){\footnotesize $\cX$}
\put(3.25,5.3){\footnotesize $\cX$}
\put(7.2,5.95){\footnotesize $\cY$}
\put(11.1,1.7){\footnotesize $\cX$}
\put(8.2,1.96){\footnotesize $\cX$}
\put(1,.4){\footnotesize $w^{\cX}_0$}
\put(14.24,.9){$u$}
\put(1.73,7.2){$v$}
\end{picture}
\caption{
Dynamics in $(u,v)$-coordinates (\ref{eq:uv}).
The origin is the point $w^{\cX}_0$ (one point of the saddle-type $\cX$-cycle).
Locally, the stable and unstable manifolds of the origin are linear and coincide with the $u$ and $v$-axes respectively.
The small circles represent points of an $\cS[k]$-cycle, where $\cS[k] = \cX^k \cY$.
For each $j = 0,\ldots,k-1$, $w^{\cS[k]}_{j n_{\cX}}$ maps to $w^{\cS[k]}_{(j+1) n_{\cX}}$ under (\ref{eq:f}) following $\cX$ (\ref{eq:gX}).
Also $w^{\cS[k]}_{k n_{\cX}}$ maps to $w^{\cS[k]}_0$ following $\cY$ (\ref{eq:gY2}).
As $k \to \infty$, the $\cS[k]$-cycle limits to the homoclinic orbit represented by small diamonds.
The four points, $a$, $b$, $c$ and $d$, of this orbit are given by (\ref{eq:abcd}).
\label{fig:uvCoords}
}
\end{center}
\end{figure}

Since $a$ maps to $c$ under $g^{\cX \cY}$, $b$ maps to $d$ under $g^{\cY \cX}$, and the homoclinic orbit is admissible,
it follows that $\Phi_i$ intersects the switching manifold whenever $(\cX \cY)_i \ne (\cY \cX)_i$.
Since $\cX_0 \ne \cY_0$, as assumed in Theorem \ref{th:codim3},
then $(\cX \cY)_i \ne (\cY \cX)_i$ for $i=0$ and for at least one other value of $i$.
Therefore $\Xi$ must have at least two intersections with the switching manifold.
The following theorem concerns the simplest scenario: that $\Xi$ has exactly two such intersections.

\begin{theorem}
Suppose (\ref{eq:f}) is invertible and satisfies the conditions of Theorem \ref{th:codim3}.
Suppose $\Xi$ intersects the switching manifold at only two points.
Then $\Phi_{n_{\cX} + n_{\cY}}$ is the line segment connecting $c$ and $d$,
and the branches of the stable and unstable manifolds of the $\cX$-cycle that involve the homoclinic orbit of part (\ref{it:HCorbit}) of Theorem \ref{th:codim3} are coincident.
\label{th:coincident}
\end{theorem}

As evident in the following proof, the assumption that $\Xi$ has only two intersections with the switching manifold
is sufficient to ensure that all points on the branch of the stable manifold involving the homoclinic orbit map to the unstable manifold.
The additional assumption of invertibility allows us to conclude that the stable and unstable branches are coincident.
Note that (\ref{eq:f}) is invertible if and only if $\delta_{\sL} \delta_{\sR} > 0$.

\begin{proof}[Proof of Theorem \ref{th:coincident}]
Since $\cX_0 \ne \cY_0$ and $\Xi$ is assumed to have only two intersections with the switching manifold,
we have $(\cX \cY)_i \ne (\cY \cX)_i$ for $i=0$ and exactly one other index in the range $i=0,\ldots,n_{\cX}+n_{\cY}-1$, call it $\alpha$.
Then there exist $\varphi_0, \psi_0 \in \Phi_0$,
such that $\varphi_0$ and $\psi_{\alpha}$ are the two points of intersection of $\Xi$ with the switching manifold
(where $\{ \varphi_i \}$ and $\{ \psi_i \}$ denote the forward orbits of $\varphi_0$ and $\psi_0$ under (\ref{eq:f})).

For ease of explanation suppose that either $\varphi_0$ lies closer to $a$ than $\psi_0$, or $\varphi_0 = \psi_0$.
(Analogous arguments produce the same result in the case that $\varphi_0$ is further from $a$ than $\psi_0$.)
Then under (\ref{eq:f}), the line segment connecting $a$ and $\varphi_0$ 
follows $\cX$ to a line segment on the $v$-axis, then follows $\cY$ to a line segment on the $u$-axis.
Thus $\varphi_{n_{\cX}+n_{\cY}}$ lies on the $u$-axis.
Similarly, under (\ref{eq:f}), the line segment connecting $b$ and $\psi_0$
follows $\cY$ to a line segment on the $u$-axis, then follows $\cX$ to a line segment elsewhere on the $u$-axis.
Thus $\psi_{n_{\cX}+n_{\cY}}$ also lies on the $u$-axis.
Since $\Phi_{n_{\cX} + n_{\cY}}$ is a piecewise-linear connection from $c$ to $d$
with possible kinks only at $\varphi_{n_{\cX} + n_{\cY}}$ and $\psi_{n_{\cX} + n_{\cY}}$, $\Phi_{n_{\cX} + n_{\cY}}$ must lie entirely on the $u$-axis.
Since (\ref{eq:f}) is assumed to be invertible, $\Phi_{n_{\cX} + n_{\cY}}$ must be the line segment connecting $c$ and $d$.

Since the homoclinic orbit is admissible,
under $n_{\cX}$ iterations of (\ref{eq:f}) both $w^{\cX}_0$ and $c$ (which lie on the $u$-axis, see Fig.~\ref{fig:uvCoords}) follow $\cX$.
Therefore the switching manifold ($x=0$) does not intersect the $u$-axis at a point between $w^{\cX}_0$ and $c$.
Therefore $\Phi_{n_{\cX} + n_{\cY}}$ is contained in the stable manifold of the $\cX$-cycle.
Furthermore, since (\ref{eq:f}) is invertible and every orbit in the branch of the unstable manifold involving the homoclinic orbit intersects $\Phi_0$,
the stable and unstable branches must be coincident.
\end{proof}

In the above proof $\varphi_0$ and $\psi_{\alpha}$ denote the intersections of $\Xi$ with the switching manifold
(where $\{ \varphi_i \}$ and $\{ \psi_i \}$ are orbits of (\ref{eq:f}) with $\varphi_0, \psi_0 \in \Phi_0$)
and it was shown that the kinks of $\Phi_{n_{\cX} + n_{\cY}}$
at $\varphi_{n_{\cX} + n_{\cY}}$ and $\psi_{n_{\cX} + n_{\cY}}$ are spurious.
This implies that $\Xi$ exhibits one of the following properties. 
Either $\Phi_0$ and $\Phi_{\alpha}$ both intersect the switching manifold
at the unique angle for which their images do not accumulate a kink, or $\varphi_0 = \psi_0$.
Since the former property corresponds to an additional codimension, for the remainder of this paper we do not consider it further.

If $\varphi_0 = \psi_0$, then the next $n_{\cX}+n_{\cY}$ iterates under (\ref{eq:f})
of any point sufficiently close to $\varphi_0$ (and not necessarily on $\Phi_0$)
follow one of four symbol sequences
depending on which side of the switching manifold the point and its image under $\alpha$ iterations of (\ref{eq:f}) are located.
These are $\cX \cY$, $\cY \cX$, $\cX^{\overline{0}} \cY$ and $\cY^{\overline{0}} \cX$,
where $\cS^{\overline{0}}$ is used to denote the word that differs from $\cS$ in only the $0^{\rm th}$ index.
The following theorem concerns $\cX^k \cY^{\overline{0}}$-cycles,
as these are saddle-type periodic solutions for each of the examples in the following section\removableFootnote{
An analogous result holds for $\cX^{\overline{0}} \cX^{k-1} \cY$-cycles,
but for the examples below such periodic solutions are virtual.
}.

\begin{theorem}
Suppose (\ref{eq:f}) is invertible and satisfies the conditions of Theorem \ref{th:codim3}.
Suppose $\Xi$ intersects the switching manifold at only two points:
$\varphi_0 \in \Phi_0$, and its $\alpha^{\rm th}$ iterate under (\ref{eq:f}), $\varphi_{\alpha}$.
Assume $g^{\cY^{\overline{0}}}$ does not map the $v$-axis to the $u$-axis.
Let $\cS'[k] = \cX^k \cY^{\overline{0}}$.
Then, as $k \to \infty$, $\cS'[k]$-cycles limit to a homoclinic orbit of the $\cX$-cycle
that has two points on the switching manifold: $\varphi_0$ and $\varphi_{\alpha}$.
\label{th:saddles}
\end{theorem}

Since $g^{\cY}$ maps the $v$-axis to the $u$ axis (part (\ref{it:gY}) of Theorem \ref{th:codim3}),
it is reasonable to assume that this is not the case for $g^{\cY^{\overline{0}}}$.
Note that Theorem \ref{th:saddles} does not ensure admissibility of $\cS'[k]$-cycles for large $k$.

\begin{proof}[Proof of Theorem \ref{th:saddles}]
In $(u,v)$-coordinates let us write $g^{\cY^{\overline{0}}}$ as\removableFootnote{
By the continuity of (\ref{eq:f}), the slopes of the partition of phase space about $\varphi_0$
and the maps under $\cX \cY$ and $\cY \cX$ completely determine the maps under
$\cX^{\overline{0}} \cY$ and $\cY^{\overline{0}} \cX$.
}
\begin{equation}
g^{\cY^{\overline{0}}}(w) = \left[ \begin{array}{cc}
\xi_{11} & \xi_{12} \\
\xi_{21} & \xi_{22}
\end{array} \right] w +
\left[ \begin{array}{c} \chi_1 \\ \chi_2 \end{array} \right] \;,
\label{eq:gY0}
\end{equation}
for some constants $\xi_{ij}$, $\chi_1$ and $\chi_2$.
We can express $\chi_1$ and $\chi_2$ in terms of other coefficients
by using the requirement that $g^{\cY}$ and $g^{\cY^{\overline{0}}}$ map $\varphi_0$ to the same point
because (\ref{eq:f}) is a continuous map.
Write $\varphi_0 = \left( 0,\hat{v} \right)$, in $(u,v)$-coordinates.
Then by (\ref{eq:gY2}) and (\ref{eq:gY0}) respectively,
\begin{equation}
g^{\cY}(\varphi_0) = \left[ \begin{array}{c} \gamma_{12} \hat{v} + \sigma_1 \\ 0 \end{array} \right] \;, \qquad
g^{\cY^{\overline{0}}}(\varphi_0) = \left[ \begin{array}{c} \xi_{12} \hat{v} + \chi_1 \\ \xi_{22} \hat{v} + \chi_2 \end{array} \right] \;.
\nonumber
\end{equation}
Therefore $\chi_1 = \sigma_1 + (\gamma_{12} - \xi_{12}) \hat{v}$ and $\chi_2 = -\xi_{22} \hat{v}$.
The assumption that $g^{\cY^{\overline{0}}}$ does not map the $v$-axis to the $u$-axis implies $\xi_{22} \ne 0$.

By composing (\ref{eq:gY0}) with $k$ instances of $g^{\cX}$ (\ref{eq:gX}),
and substituting $\lambda_2 = \frac{1}{\lambda_1}$ and the above formulas for $\chi_1$ and $\chi_2$, we obtain
\begin{equation}
g^{\cY^{\overline{0}} \cX^k}(w) =
\left[ \begin{array}{cc} \xi_{11} \lambda_1^k & \xi_{12} \lambda_1^k \\
\frac{\xi_{21}}{\lambda_1^k} & \frac{\xi_{22}}{\lambda_1^k} \end{array} \right] w +
\left[ \begin{array}{c}
\left( \sigma_1 + \left( \gamma_{12} - \xi_{12} \right) \hat{v} \right) \lambda_1^k \\
-\frac{\xi_{22} \hat{v}}{\lambda_1^k}
\end{array} \right] \;.
\label{eq:gY0Xk}
\end{equation}
Since $\xi_{22} \ne 0$, for large $k$ (\ref{eq:gY0Xk}) has the unique fixed point
\begin{equation}
w^{\cS'[k]}_{k n_{\cX}} =
\left[ \begin{array}{c}
\left( \sigma_1 + \gamma_{12} \hat{v} \right) \lambda_1^k + \cO \left( \lambda_1^{2k} \right) \\
\hat{v} + \cO \left( \lambda_1^k \right)
\end{array} \right] \;.
\nonumber
\end{equation}
Therefore $w^{\cS'[k]}_{k n_{\cX}} \to \varphi_0$, as $k \to \infty$, and also $w^{\cS'[k]}_{k n_{\cX} + \alpha} \to \varphi_{\alpha}$.
\end{proof}

\section{Finding infinitely many attracting periodic solutions}
\label{sec:finding}

This section introduces a practical method to finding parameter values
$\tau_{\sL}$, $\delta_{\sL}$, $\tau_{\sR}$ and $\delta_{\sR}$
for which (\ref{eq:f}) has infinitely many attracting periodic solutions.
The method is then applied to produce three examples.

By Theorem \ref{th:codim3}, three requirements necessary for infinite coexistence are
$\lambda_2 = \frac{1}{\lambda_1}$, $\gamma_{22} = 0$ and $\sigma_2 = 0$,
where $\lambda_1$ and $\lambda_2$ are the eigenvalues of $M_{\cX}$, and $\gamma_{22}$ and $\sigma_2$ are coefficients of the map $g^{\cY}$ (\ref{eq:gY}).
In order to identify suitable values of $\tau_{\sL}$, $\delta_{\sL}$, $\tau_{\sR}$ and $\delta_{\sR}$,
we translate these requirements into three restrictions on the parameter values.

We begin with the requirement $\lambda_2 = \frac{1}{\lambda_1}$, which is equivalent to $\det \left( M_{\cX} \right) = 1$.
Let $l_{\cX}$ denote the number of $\sL$'s that are present in the word $\cX$.
Then $M_{\cX}$ is a product of $l_{\cX}$ instances of $A_{\sL}$, and $n_{\cX} - l_{\cX}$ instances of $A_{\sR}$, hence
$\det \left( M_{\cX} \right) = \delta_{\sL}^{l_{\cX}} \delta_{\sR}^{n_{\cX} - l_{\cX}}$.
Therefore $\lambda_2 = \frac{1}{\lambda_1}$ is equivalent to
\begin{equation}
\delta_{\sL} = \delta_{\sR}^{-\frac{n_{\cX} - l_{\cX}}{l_{\cX}}} \;.
\label{eq:deltaL}
\end{equation}

It is impractical to directly impose the remaining two requirements, $\gamma_{22} = 0$ and $\sigma_2 = 0$,
because expressions for $\gamma_{22}$ and $\sigma_2$
in terms of the four parameters are extremely complicated even for simple choices of $\cX$ and $\cY$.
This is because $\gamma_{22}$ and $\sigma_2$ are coefficients in $(u,v)$-coordinates (\ref{eq:uv}),
and consequently expressions for these coefficients involve formulas for the eigenvalues of $M_{\cX}$, which involve a square root.
Algebraic manipulations are made substantially more manageable by instead using the results of \S\ref{sec:further} to
express $\zeta_1$ and $\zeta_2$ (eigenvectors of $M_{\cX}$) in terms of the parameters.
This is explained below.
We then let
\begin{equation}
Q = \big[ \zeta_1 \; \zeta_2 \big] \;, \qquad
\Omega = Q^{-1} M_{\cX} Q \;.
\label{eq:QOmega}
\end{equation}
If $\zeta_1$ and $\zeta_2$ are linearly independent eigenvectors of $M_{\cX}$, then the matrix $\Omega$ must be diagonal.
That is,
\begin{equation}
\omega_{12} = 0 \;, \qquad
\omega_{21} = 0 \;,
\label{eq:omega1221}
\end{equation}
where $\omega_{ij}$ denotes the $(i,j)$-element of $\Omega$.
The equations of (\ref{eq:omega1221}) represent an alternative to $\gamma_{22} = 0$ and $\sigma_2 = 0$
that are significantly simpler when expressed in terms of $\tau_{\sL}$, $\delta_{\sL}$, $\tau_{\sR}$ and $\delta_{\sR}$.

To obtain expressions for $\zeta_1$ and $\zeta_2$,
let us assume that $(\cX \cY)_i \ne (\cY \cX)_i$ only for $i=0$ and $i=\alpha$, and $\varphi_0 = \psi_0$ (see \S\ref{sec:further}).
The point $\varphi_0$ lies on the switching manifold, and its image under $\alpha$ iterations of (\ref{eq:f}) following $\cX \cY$ also lies on the switching manifold.
In $(x,y)$-coordinates let us write $\varphi_0 = (0,\hat{y})$.
The value of $\hat{y}$ may be determined from the requirement that the $x$-component of $\varphi_{\alpha}$ is zero.
Also, $\varphi_0$ lies on the unstable manifold of $\left( x^{\cX}_0, y^{\cX}_0 \right)$. 
The point $\left( x^{\cX}_0, y^{\cX}_0 \right)$ is given by (\ref{eq:zStarX}),
and its unstable manifold has direction $\zeta_2$.
Therefore, when $\mu = 1$, there exists $\eta \in \mathbb{R}$ such that
\begin{equation}
\left( I - M_{\cX} \right)^{-1} P_{\cX} \left[ \begin{array}{c} 1 \\ 0 \end{array} \right] +
\eta \zeta_2 = \left[ \begin{array}{c} 0 \\ \hat{y} \end{array} \right] \;.
\label{eq:unstabIntCritPoint}
\end{equation}
By using $M_{\cX} \zeta_2 = \lambda_2 \zeta_2$, we may rearrange (\ref{eq:unstabIntCritPoint}) to obtain
\begin{equation}
\eta \left( 1 - \lambda_2 \right) \zeta_2 =
\left( I - M_{\cX} \right) \left[ \begin{array}{c} 0 \\ \hat{y} \end{array} \right] -
P_{\cX} \left[ \begin{array}{c} 1 \\ 0 \end{array} \right] \;.
\label{eq:unstabIntCritPoint2}
\end{equation}
Since we are free to choose the magnitude of $\zeta_2$, by (\ref{eq:unstabIntCritPoint2}) we may set
\begin{equation}
\zeta_2 =
\left( I - M_{\cX} \right) \left[ \begin{array}{c} 0 \\ \hat{y} \end{array} \right] -
P_{\cX} \left[ \begin{array}{c} 1 \\ 0 \end{array} \right] \;.
\label{eq:zeta2}
\end{equation}
In $(x,y)$-coordinates, the $u$ and $v$-axes have the same directions as $\zeta_1$ and $\zeta_2$, respectively.
Part (\ref{it:gY}) of Theorem \ref{th:codim3} tells us that $g^{\cY}$ maps the $v$-axis to the $u$-axis.
Therefore, $\zeta_1$ is a scalar multiple of $M_{\cY} \zeta_2$.
We could set $\zeta_1 = M_{\cY} \zeta_2$, but for the examples below it is more convenient to set
\begin{equation}
\zeta_1 = M_{\cX}^{-1} M_{\cY} \zeta_2 \;.
\label{eq:zeta1}
\end{equation}

In summary, (\ref{eq:deltaL}) and (\ref{eq:omega1221})
represent three restrictions on the parameter values of (\ref{eq:f})
for which the map has infinitely many admissible, stable $\cS[k]$-cycles,
where $\omega_{12}$ and $\omega_{21}$ are the off-diagonal elements of $\Omega$ (\ref{eq:QOmega}),
and $\zeta_1$ and $\zeta_2$ are given by (\ref{eq:zeta2}) and (\ref{eq:zeta1}).
We now find solutions to (\ref{eq:deltaL}) and (\ref{eq:omega1221}) for three different combinations of $\cX$ and $\cY$.

\subsubsection*{An example with $n_{\cX} = 3$}

Suppose\removableFootnote{
{\sc findExampleF4.m}
}
\begin{equation}
\cX = \sR \sL \sR \;, \qquad
\cY = \sL \sR \;,
\label{eq:XYRLRkLR}
\end{equation}
as in Fig.~\ref{fig:infFa}.
Here $\cX$ has one $\sL$ and three symbols total, i.e.~$l_{\cX} = 1$ and $n_{\cX} = 3$.
Thus by (\ref{eq:deltaL}), $\delta_{\sL} = \frac{1}{\delta_{\sR}^2}$.

Also $\alpha = 1$ (because $\cX \cY = \sR \sL \sR \sL \sR$ and $\cY \cX = \sL \sR \sR \sL \sR$, thus $(\cX \cY)_i \ne (\cY \cX)_i$ only for $i=0$ and $i=1$).
Hence we require that $\varphi_0 = (0,\hat{y})$ maps to the switching manifold under a single iteration of (\ref{eq:f}).
This implies $\hat{y} = -1$, when $\mu = 1$, with which (\ref{eq:zeta2}) gives
$\zeta_2 = \left[ -\tau_{\sR}-1 ,\; \delta_{\sR}-1 \right]^{\sf T}$,
and (\ref{eq:zeta1}) gives
$\zeta_1 = A_{\sR}^{-1} \zeta_2 = \left[ \frac{1}{\delta_{\sR}}-1 ,\; -\frac{\tau_{\sR}}{\delta_{\sR}}-1 \right]^{\sf T}$.
By substituting these into (\ref{eq:QOmega}) we obtain
\begin{align}
\omega_{12} &= \frac{1}
{\delta_{\sR} \left( \delta_{\sR}^2 + \tau_{\sR} \delta_{\sR} - \delta_{\sR} + \tau_{\sR}^2 + \tau_{\sR} + 1 \right)}
\Big( -\tau_{\sL} \delta_{\sR}^3 + \tau_{\sR} \delta_{\sR}^3 - \tau_{\sL} \tau_{\sR}^2 \delta_{\sR}^3
+ \tau_{\sR}^2 \delta_{\sR}^3 + \tau_{\sR} - \tau_{\sR}^2 \delta_{\sR} \nonumber \\
&\quad- \tau_{\sR} \delta_{\sR} - 2 \delta_{\sR}
+ \tau_{\sR}^2 + 1 - \tau_{\sL} \tau_{\sR} \delta_{\sR}^2 + \tau_{\sL} \delta_{\sR}^4 + \tau_{\sR} \delta_{\sR}^4
+ \delta_{\sR}^4 - \tau_{\sL} \tau_{\sR}^2 \delta_{\sR}^2 - \tau_{\sL} \tau_{\sR}^3 \delta_{\sR}^2 + \delta_{\sR}^2 \Big) \;,
\label{eq:omega12F} \\
\omega_{21} &= \frac{1}
{\delta_{\sR}^2 \left( \delta_{\sR}^2 + \tau_{\sR} \delta_{\sR} - \delta_{\sR} + \tau_{\sR}^2 + \tau_{\sR} + 1 \right)}
\Big( -\tau_{\sL} \delta_{\sR}^3 + 2 \delta_{\sR}^4 + \tau_{\sR} \delta_{\sR}^3
+ \tau_{\sL} \delta_{\sR}^4 - \tau_{\sR} \delta_{\sR}^4 + \tau_{\sL} \tau_{\sR}^2 \delta_{\sR}^3 - \delta_{\sR}^5 \nonumber \\
&\quad- \tau_{\sR}^2 \delta_{\sR}^3 - \tau_{\sR} \delta_{\sR} - \delta_{\sR}
+ \tau_{\sL} \tau_{\sR}^2 \delta_{\sR}^2 + \tau_{\sL} \tau_{\sR}^3 \delta_{\sR}^2 + \tau_{\sR}^2 \delta_{\sR}^2
- \tau_{\sR}^2 - \tau_{\sR} + \tau_{\sL} \tau_{\sR} \delta_{\sR}^4 - \delta_{\sR}^3 \Big) \;.
\label{eq:omega21F}
\end{align}
We wish to solve $\omega_{12} = \omega_{21} = 0$.
To this end we notice that the sum of (\ref{eq:omega12F}) and (\ref{eq:omega21F}) factors conveniently:
\begin{equation}
\omega_{12} + \omega_{21} = \frac{\left( \delta_{\sR} - 1 \right)
\left( \tau_{\sR} \delta_{\sR}^3 + \tau_{\sL} \delta_{\sR}^3 - \tau_{\sL} \tau_{\sR}^2 \delta_{\sR}^2 + \tau_{\sR} + 2 \delta_{\sR}^2 \right)
\left( \delta_{\sR} + \tau_{\sR} + 1 \right)}
{\delta_{\sR}^2 \left( \delta_{\sR}^2 + \tau_{\sR} \delta_{\sR} - \delta_{\sR} + \tau_{\sR}^2 + \tau_{\sR} + 1 \right)} \;.
\label{eq:sumomega1221}
\end{equation}

The first factor in the numerator of (\ref{eq:sumomega1221}) is zero when $\delta_{\sR} = 1$\removableFootnote{
with which we find that $\omega_{12} = 0$ when $\tau_{\sL} = \frac{1}{\tau_{\sR}}$.
}.
For any $\tau_{\sR} < -1$, this combination of parameter values gives infinitely many admissible, stable $\cS[k]$-cycles,
but the $\cS[k]$-cycles are not asymptotically stable because the eigenvalues of $M_{\cS[k]}$ lie on the unit circle.
In this case (\ref{eq:f}) is area-preserving and the $\cS[k]$-cycles are elliptic.

The second factor in the numerator of (\ref{eq:sumomega1221}) is zero when 
$\tau_{\sL} = \frac{\tau_{\sR} \delta_{\sR}^3 + \tau_{\sR} + 2 \delta_{\sR}^2}
{\delta_{\sR}^2 \left( \tau_{\sR}^2 - \delta_{\sR} \right)}$.
However, we then have $\omega_{12} = \frac{\det(Q) \delta_{\sR}}{\tau_{\sR}^2 - \delta_{\sR}}$,
which cannot be zero because $Q$ must be non-singular.

Finally, the third factor in the numerator of (\ref{eq:sumomega1221}) is zero when $\tau_{\sR} = -1-\delta_{\sR}$.
Then $\omega_{12} = 0$ when $\tau_{\sL} = -1 + \frac{1}{\delta_{\sR}} - \frac{1}{\delta_{\sR}^2 \left( \delta_{\sR}^2 + 1 \right)}$.
Therefore, $\delta_{\sR}$ is undetermined and
\begin{equation}
\tau_{\sL} = -1 + \frac{1}{\delta_{\sR}} - \frac{1}{\delta_{\sR}^2 (\delta_{\sR}^2+1)} \;, \qquad
\delta_{\sL} = \frac{1}{\delta_{\sR}^2} \;, \qquad
\tau_{\sR} = -1 - \delta_{\sR} \;.
\label{eq:paramRLRkLR}
\end{equation}
With (\ref{eq:paramRLRkLR}), $\delta_{\sR} > 1$ and $\mu = 1$,
(\ref{eq:f}) indeed has infinitely many admissible, attracting $\cS[k]$-cycles.
This is proved in \S\ref{sec:sufficient}.
Fig.~\ref{fig:infFa} illustrates this scenario with $\delta_{\sR} = \frac{3}{2}$.
For different values of $\delta_{\sR} > 1$ the primary features of the phase portrait are unchanged.

\subsubsection*{An example with $n_{\cX} = 4$}

Suppose\removableFootnote{
{\sc findExampleI2.m}
}
\begin{equation}
\cX = \sR \sL \sL \sR \;, \qquad
\cY = \sL \sL \sR \;.
\label{eq:XYRLLRkLLR}
\end{equation}
Then (\ref{eq:deltaL}) gives $\delta_{\sL} = \frac{1}{\delta_{\sR}}$.
Also $\alpha = 2$, which implies $\hat{y} = -1-\frac{1}{\tau_{\sL}}$, when $\mu = 1$.
By continuing with the above method we find that expressions for $\omega_{12}$ and $\omega_{21}$ are too complicated to include here
-- and the author has been unable to solve (\ref{eq:deltaL}) and (\ref{eq:omega1221}) analytically for this example --
but with $\tau_{\sL} = 0.5$, (\ref{eq:deltaL}) and (\ref{eq:omega1221}) admit the following approximate numerical solution
\begin{equation}
\tau_{\sL} = 0.5 \;, \qquad
\delta_{\sL} = \frac{1}{\delta_{\sR}} \;, \qquad
\tau_{\sR} = -1.139755486 \;, \qquad
\delta_{\sR} = 1.378851759 \;.
\label{eq:paramI}
\end{equation}
A phase portrait of (\ref{eq:f}) with $\mu = 1$ and (\ref{eq:paramI}) is shown in Fig.~\ref{fig:infI}.
Here at least eight $\cS[k]$-cycles are admissible and attracting, which suggests that 
with the exact solution to (\ref{eq:deltaL}) and (\ref{eq:omega1221}) 
infinitely many $\cS[k]$-cycles are admissible and attracting.

\begin{figure}[t!]
\begin{center}
\setlength{\unitlength}{1cm}
\begin{picture}(15,7.5)
\put(0,0){\includegraphics[height=7.5cm]{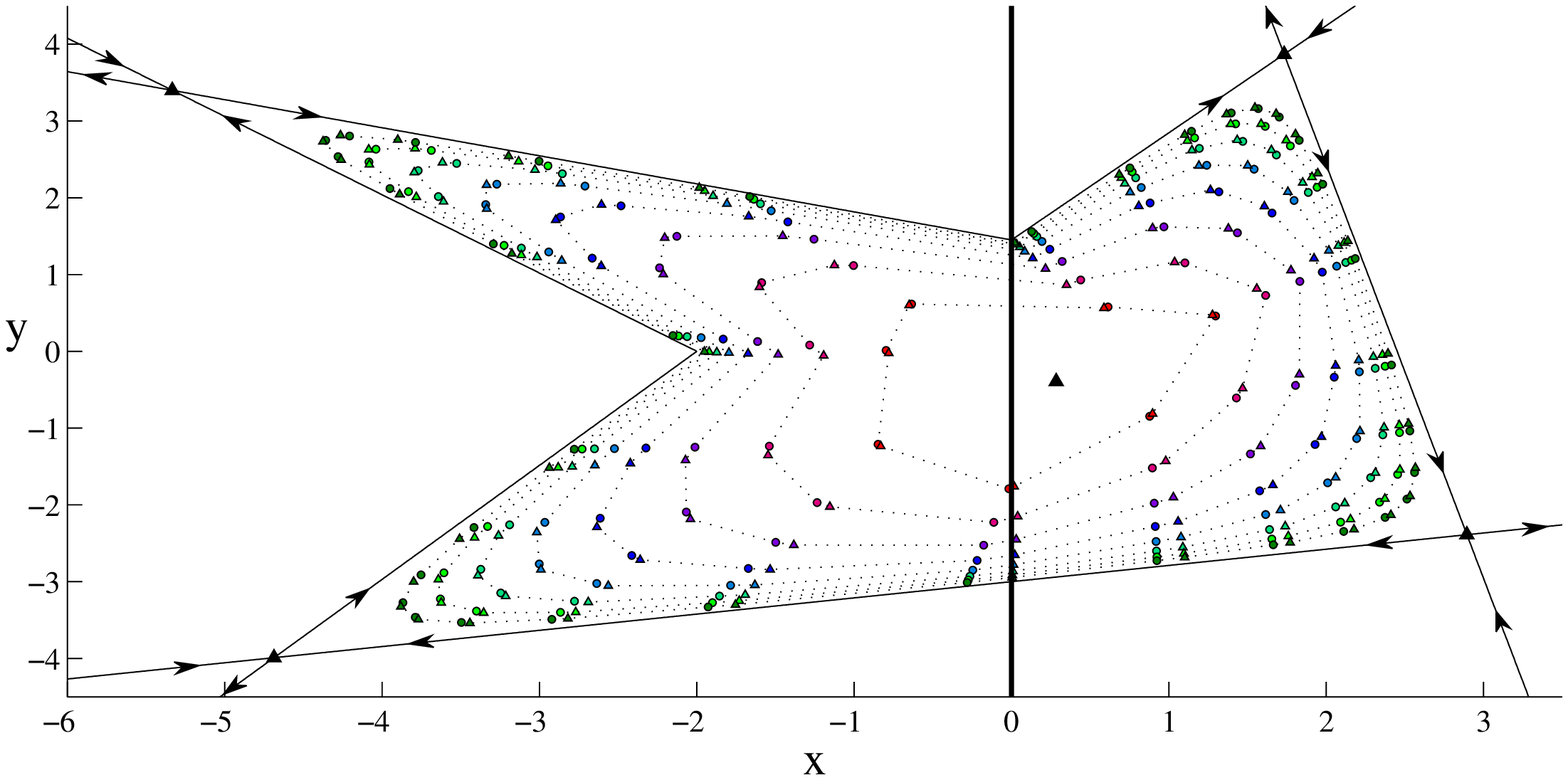}}
\end{picture}
\caption{
A phase portrait of (\ref{eq:f}) with $\mu = 1$ and (\ref{eq:paramI}).
These are approximate parameter values for the infinite coexistence of admissible, attracting $\cS[k]$-cycles,
where $\cS[k] = \cX^k \cY$, and $\cX = \sR \sL \sL \sR$ and $\cY = \sL \sL \sR$ (\ref{eq:XYRLLRkLLR}).
The $\cX$-cycle is of saddle-type and its stable and unstable manifolds are indicated.
As in Fig.~\ref{fig:infFa}, $\cS[k]$-cycles are indicated by small circles for $k=1$ to $k=8$,
and $\cS'[k]$-cycles (where $\cS'[k] = \cX^k \cY^{\overline{0}}$)
are indicated by small triangles for the same values of $k$.
\label{fig:infI}
}
\end{center}
\end{figure}

\subsubsection*{An example with $n_{\cX} = 5$}

Suppose\removableFootnote{
{\sc findExampleC3.m}
}
\begin{equation}
\cX = \sR \sL \sR \sL \sR \;, \qquad
\cY = \sL \sR \;.
\label{eq:XYRLRLRkLR}
\end{equation}
Here (\ref{eq:deltaL}) gives $\delta_{\sL} = \delta_{\sR}^{-\frac{3}{2}}$.
Also $\alpha = 1$, thus $\hat{y} = -1$ when $\mu = 1$.
As with the previous example, it does not appear to be possible to solve (\ref{eq:deltaL}) and (\ref{eq:omega1221}) analytically.
An approximate numerical solution to (\ref{eq:deltaL}) and (\ref{eq:omega1221}) is
\begin{equation}
\tau_{\sL} = -0.7 \;, \qquad
\delta_{\sL} = \delta_{\sR}^{-\frac{3}{2}} \;, \qquad 
\tau_{\sR} = -3.308423793 \;, \qquad
\delta_{\sR} = 1.659870677 \;,
\label{eq:paramC}
\end{equation}
which is illustrated in Fig.~\ref{fig:infC}.
As with the previous example, from Fig.~\ref{fig:infC} we infer that with the exact solution to
(\ref{eq:deltaL}) and (\ref{eq:omega1221}), 
(\ref{eq:f}) has infinitely many admissible, attracting $\cS[k]$-cycles.

\begin{figure}[t!]
\begin{center}
\setlength{\unitlength}{1cm}
\begin{picture}(15,7.5)
\put(0,0){\includegraphics[height=7.5cm]{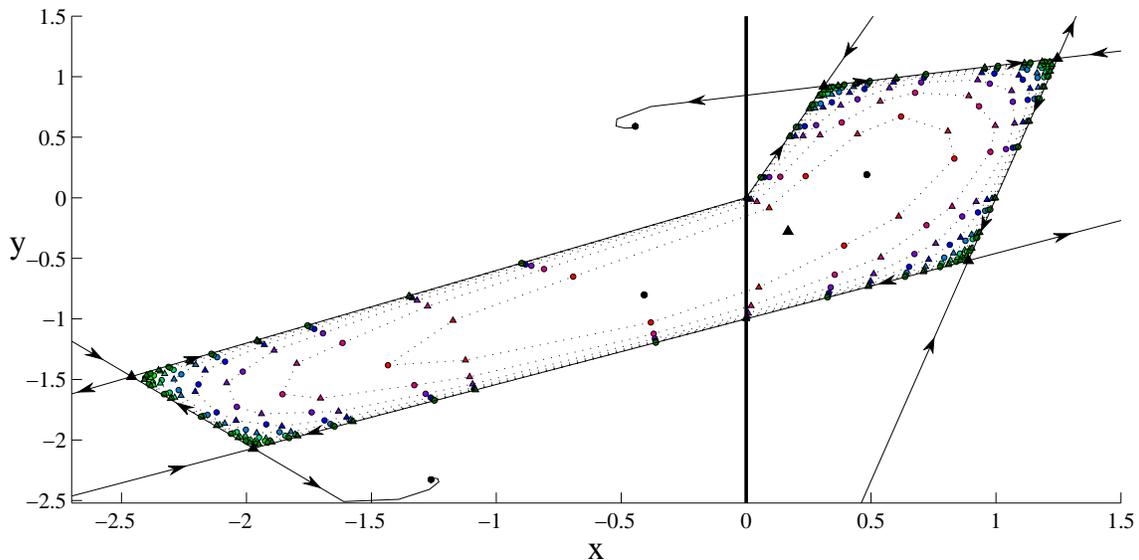}}
\end{picture}
\caption{
A phase portrait of (\ref{eq:f}) with $\mu = 1$ and (\ref{eq:paramC}), using the same conventions as Figs.~\ref{fig:infFa} and \ref{fig:infI}.
The parameter values approximate those admitting infinite coexistence, with $\cX$ and $\cY$ given by (\ref{eq:XYRLRLRkLR}).
$\cS[k]$ and $\cS'[k]$-cycles are plotted for $k=1$ to $k=8$.
For the given parameters, the map (\ref{eq:f}) also has
an unstable fixed point (the isolated triangle),
an attracting $\sR \sL$-cycle (the two isolated circles near the middle of the figure),
and an attracting $\sR \sL \sR \sL \sL$-cycle (two points of this periodic solution are visible in the figure).
\label{fig:infC}
}
\end{center}
\end{figure}

\section{Verification of infinite coexistence}
\label{sec:sufficient}

We have shown that if the map (\ref{eq:f}) with $\mu = 1$ has infinitely many admissible, attracting $\cS[k]$-cycles
with $\cX = \sR \sL \sR$ and $\cY = \sL \sR$,
then, with reasonable assumptions given in sections \ref{sec:necessary} and \ref{sec:further},
the parameter values must satisfy (\ref{eq:paramRLRkLR}).
In this section it is shown that the additional restriction, $\delta_{\sR} > 1$, is sufficient for (\ref{eq:f}) to exhibit such infinite coexistence,
as indicated in the following theorem.

\begin{theorem}
Let $\cS[k] = \left( \sR \sL \sR \right)^k \sL \sR$, and $\cS'[k] = \left( \sR \sL \sR \right)^k \sR \sR$.
Let $\mu = 1$, $\delta_{\sR} > 1$, and suppose that the remaining parameter values of (\ref{eq:f}) are given by (\ref{eq:paramRLRkLR}).
Then for all $k \ge 1$,
(\ref{eq:f}) has a unique, admissible, asymptotically stable $\cS[k]$-cycle,
and a unique, admissible, saddle-type $\cS'[k]$-cycle.
\label{th:RLRkLR}
\end{theorem}

The theorem is proved below by directly verifying
asymptotic stability of the $\cS[k]$-cycles, and admissibility of the $\cS[k]$ and $\cS'[k]$-cycles.
By the results of \S\ref{sec:periodic}, this may be done by calculating the determinant and trace of $M_{\cS[k]}$ and $M_{\cS'[k]}$,
and the determinant of each $P_{\cS[k]^{(i)}}$ and $P_{\cS'[k]^{(i)}}$.
Since the details of these calculations are relatively lengthy and involve significant repetition,
for brevity the majority of the calculations for the $\cS'[k]$-cycles are omitted.
We begin by computing the determinant and trace of $M_{\cS[k]}$.

\begin{lemma}
Let $\cS[k] = \left( \sR \sL \sR \right)^k \sL \sR$,
suppose $\delta_{\sR} \ne 0$ and that the parameter values of (\ref{eq:f}) satisfy (\ref{eq:paramRLRkLR}).
Then for all $k \ge 1$,
\begin{align}
\det \left( M_{\cS[k]} \right) &= \frac{1}{\delta_{\sR}} \;, \label{eq:detMS} \\
{\rm trace} \left( M_{\cS[k]} \right) &= \frac{-(\delta_{\sR} + 1) \lambda_1^k}{\delta_{\sR}^2 + 1} \;, \label{eq:traceMS}
\end{align}
where
\begin{equation}
\lambda_1 = \frac{\delta_{\sR}}{\delta_{\sR}^2+1} \;.
\label{eq:lambda1}
\end{equation}
\label{le:MS}
\end{lemma}

\begin{proof} 
Here $M_{\cS[k]}$ is the product of $k+1$ instances of $A_{\sL}$, and $2k+1$ instances of $A_{\sR}$,
therefore $\det \left( M_{\cS[k]} \right) = \delta_{\sL}^{k+1} \delta_{\sR}^{2k+1}$.
By substituting $\delta_{\sL} = \frac{1}{\delta_{\sR}^2}$ into this expression we obtain (\ref{eq:detMS}).

More effort is required to compute ${\rm trace} \left( M_{\cS[k]} \right)$.
By the definitions of $\cS[k]$ and $M_{\cS}$ (\ref{eq:MSPS}), we can write $M_{\cS[k]} = A_{\sR} A_{\sL} M_{\sR \sL \sR}^k$.
An evaluation of $M_{\sR \sL \sR} = A_{\sR} A_{\sL} A_{\sR}$ using (\ref{eq:f}) and (\ref{eq:paramRLRkLR}) produces
\begin{equation}
M_{\sR \sL \sR} = \frac{1}{\delta_{\sR}^2 + 1} \left[ \begin{array}{cc}
(\delta_{\sR} + 1)^2 & \delta_{\sR}^3 - 1 \\
\delta_{\sR}^2 - \frac{1}{\delta_{\sR}} & (\delta_{\sR}^2 + 1)(\delta_{\sR} - 1) + \frac{1}{\delta_{\sR}}
\end{array} \right] \;.
\end{equation}
The matrix $M_{\sR \sL \sR}$ has eigenvalues $\lambda_1$ (\ref{eq:lambda1}) and $\lambda_2 = \frac{1}{\lambda_1}$.
To take powers of $M_{\sR \sL \sR}$ we let 
$Q = \left[ \begin{array}{cc}
1 & \frac{\delta_{\sR}}{\delta_{\sR} - 1} \\
\frac{-1}{\delta_{\sR} - 1} & 1
\end{array} \right]$,
as the columns of this matrix are eigenvectors of $M_{\sR \sL \sR}$.
Then $M_{\sR \sL \sR}^k = Q^{-1} \left[ \begin{array}{cc} \lambda_1^k & 0 \\ 0 & \frac{1}{\lambda_1^k} \end{array} \right] Q$, and consequently
\begin{equation}
M_{\sR \sL \sR}^k = \frac{1}{1 + \frac{\delta_{\sR}}{(\delta_{\sR} - 1)^2}} \left[ \begin{array}{cc}
\lambda_1^k + \frac{\delta_{\sR}}{(\delta_{\sR}-1)^2} \lambda_1^{-k} &
\frac{-\delta_{\sR}}{\delta_{\sR}-1} \left( \lambda_1^k - \lambda_1^{-k} \right) \\
\frac{-1}{\delta_{\sR}-1} \left( \lambda_1^k - \lambda_1^{-k} \right) &
\frac{\delta_{\sR}}{(\delta_{\sR}-1)^2} \lambda_1^k + \lambda_1^{-k}
\end{array} \right] \;.
\label{eq:MRLRk}
\end{equation}
An evaluation of ${\rm trace} \left( A_{\sR} A_{\sL} M_{\sR \sL \sR}^k \right)$ using (\ref{eq:MRLRk}) produces (\ref{eq:traceMS}).
\end{proof}

Next we derive expressions for the determinant of each $P_{\cS[k]^{(i)}}$.
Since $\cS[k]$ has period $3k+2$, we require $\det \left( P_{\cS[k]^{(i)}} \right)$ for each $i = 0, \ldots, 3k+1$.

\begin{lemma}
Let $\cS[k] = \left( \sR \sL \sR \right)^k \sL \sR$,
suppose $\delta_{\sR} \ne 0$ and that the parameter values of (\ref{eq:f}) satisfy (\ref{eq:paramRLRkLR}).
Then for all $j = 0,\ldots,k-1$,
\begin{align}
\det \left( P_{\cS[k]^{(3j)}} \right) &= \frac{1}{\delta_{\sR}^2-\delta_{\sR}+1} \left(
\delta_{\sR}(\delta_{\sR} + 1) +
\frac{\delta_{\sR}^2(\delta_{\sR}+1)}{\delta_{\sR}^2 + 1} \lambda_1^k -
(\delta_{\sR}^2+\delta_{\sR}+1) \lambda_1^{k-j} -
\frac{\delta_{\sR}^3-1}{\delta_{\sR}^2 + 1} \lambda_1^j \right) \;, \label{eq:Ps3j} \\
\det \left( P_{\cS[k]^{(3j+1)}} \right) &= -\frac{1}{\delta_{\sR}^2-\delta_{\sR}+1} \bigg(
(\delta_{\sR}+1)(\delta_{\sR}^2+1) +
\delta_{\sR} (\delta_{\sR}+1) \lambda_1^k \nonumber \\
&\quad-\frac{(\delta_{\sR}^2+1)(\delta_{\sR}^2+\delta_{\sR}+1)}{\delta_{\sR}} \lambda_1^{k-j} -
\frac{\delta_{\sR}^2(\delta_{\sR}^2+\delta_{\sR}+1)}{\delta_{\sR}^2+1} \lambda_1^j \bigg) \;, \label{eq:Ps3jp1} \\
\det \left( P_{\cS[k]^{(3j+2)}} \right) &= \frac{1}{\delta_{\sR}^2-\delta_{\sR}+1} \bigg(
\frac{\delta_{\sR}+1}{\delta_{\sR}^2} +
\frac{\delta_{\sR}+1}{\delta_{\sR} (\delta_{\sR}^2+1)} \lambda_1^k \nonumber \\
&\quad+\frac{(\delta_{\sR}-1)(\delta_{\sR}^2+\delta_{\sR}+1)(\delta_{\sR}^2+1)}{\delta_{\sR}^3} \lambda_1^{k-j} -
\frac{\delta_{\sR}^2+\delta_{\sR}+1}{(\delta_{\sR}^2+1)^2} \lambda_1^j \bigg) \;, \label{eq:Ps3jp2} \\
\det \left( P_{\cS[k]^{(3k)}} \right) &= -\frac{1 - \lambda_1^k}{\delta_{\sR}^2 - \delta_{\sR} + 1} \;, \label{eq:Ps3k} \\
\det \left( P_{\cS[k]^{(3k+1)}} \right) &=
\frac{(\delta_{\sR}^5 + \delta_{\sR}^3 + \delta_{\sR} - 1) \lambda_1^k + 1}
{\delta_{\sR}^2 (\delta_{\sR}^2 + 1) (\delta_{\sR}^2 - \delta_{\sR} + 1)} \;, \label{eq:Ps3kp1}
\end{align}
where $\lambda_1$ is given by (\ref{eq:lambda1}).
\label{le:detPSi}
\end{lemma}

\begin{proof}
Here we derive only (\ref{eq:Ps3j}).
Derivations of (\ref{eq:Ps3jp1}) and (\ref{eq:Ps3jp2}) are similar;
derivations of (\ref{eq:Ps3k}) and (\ref{eq:Ps3kp1}) are simpler.

Taking the left shift permutation of $\cS[k]$ a total of $3j$ places yields
\begin{equation}
\cS[k]^{(3j)} = \left( \sR \sL \sR \right)^{k-j} \sL \sR \left( \sR \sL \sR \right)^j \;.
\nonumber
\end{equation}
Careful use of (\ref{eq:MSPS}) produces\removableFootnote{
Also
\begin{align}
\cS[k]^{(3j+1}) &= \sL \sR \left( \sR \sL \sR \right)^{k-j-1} \sL \sR \left( \sR \sL \sR \right)^j \sR \;, \\
\cS[k]^{(3j+2}) &= \sR \left( \sR \sL \sR \right)^{k-j-1} \sL \sR \left( \sR \sL \sR \right)^j \sR \sL \;, \\
\cS[k]^{(3k}) &= \sL \sR \left( \sR \sL \sR \right)^k \;, \\
\cS[k]^{(3k+1}) &= \sR \left( \sR \sL \sR \right)^k \sL \;, \\
\end{align}
and thus
\begin{align}
P_{\cS[k]^{(3j+1)}} &= I + A_{\sR} \left( \sum_{p=0}^{j-1} M_{\sR \sL \sR}^p \right)
\left( I + A_{\sR} + A_{\sR} A_{\sL} \right) \nonumber \\
&\quad+A_{\sR} M_{\sR \sL \sR}^j
\left( I + A_{\sR} + A_{\sR} A_{\sL}
\left( \sum_{p=0}^{k-j-2} M_{\sR \sL \sR}^p \right)
\left( I + A_{\sR} + A_{\sR} A_{\sL} \right) +
M_{\sR \sL \sR}^{k-j-1} \left( I + A_{\sR} \right) \right) \;, \\
P_{\cS[k]^{(3j+2)}} &= I + A_{\sL} + A_{\sL} A_{\sR} \left( \sum_{p=0}^{j-1} M_{\sR \sL \sR}^p \right)
\left( I + A_{\sR} + A_{\sR} A_{\sL} \right) \nonumber \\
&\quad+A_{\sL} A_{\sR} M_{\sR \sL \sR}^j
\left( I + A_{\sR} + A_{\sR} A_{\sL}
\left( \sum_{p=0}^{k-j-2} M_{\sR \sL \sR}^p \right)
\left( I + A_{\sR} + A_{\sR} A_{\sL} \right) +
M_{\sR \sL \sR}^{k-j-1} \right) \;, \\
P_{\cS[k]^{(3k)}} &= \left( \sum_{p=0}^{k-1} M_{\sR \sL \sR}^p \right)
\left( I + A_{\sR} + A_{\sR} A_{\sL} \right) +
M_{\sR \sL \sR}^k \left( I + A_{\sR} \right) \;, \\
P_{\cS[k]^{(3k+1)}} &= I + A_{\sL} \left( \sum_{p=0}^{k-1} M_{\sR \sL \sR}^p \right)
\left( I + A_{\sR} + A_{\sR} A_{\sL} \right) +
A_{\sL} M_{\sR \sL \sR}^k \;.
\end{align}
}
\begin{align}
P_{\cS[k]^{(3j)}} &= \left( \sum_{p=0}^{j-1} M_{\sR \sL \sR}^p \right)
\left( I + A_{\sR} + A_{\sR} A_{\sL} \right) \nonumber \\
&\quad+M_{\sR \sL \sR}^j
\left( I + A_{\sR} + A_{\sR} A_{\sL}
\left( \sum_{p=0}^{k-j-1} M_{\sR \sL \sR}^p \right)
\left( I + A_{\sR} + A_{\sR} A_{\sL} \right) \right) \;.
\label{eq:PS3j}
\end{align}
Powers of $M_{\sR \sL \sR}$ are given by (\ref{eq:MRLRk}).
To obtain explicit expressions for the two finite series that appear in (\ref{eq:PS3j}),
we use the following formulas for the partial sum of a geometric series:
\begin{equation}
\sum_{p=0}^{j-1} \lambda_1^p = \frac{1 - \lambda_1^j}{1 - \lambda_1} \;, \qquad
\sum_{p=0}^{j-1} \lambda_1^{-p} = \frac{\lambda_1(\lambda_1^{-j}-1)}{1 - \lambda_1} \;.
\nonumber
\end{equation}
This gives
\begin{equation}
\sum_{p=0}^{j-1} M_{\sR \sL \sR}^p = \frac{1}
{\left( 1 + \frac{\delta_{\sR}}{(\delta_{\sR} - 1)^2} \right) \left( 1 - \lambda_1 \right)}
\left[ \begin{array}{cc}
1-\lambda_1^j + \frac{\delta_{\sR}}{(\delta_{\sR}-1)^2} \lambda_1(\lambda_1^{-j}-1)&
\frac{-\delta_{\sR}}{\delta_{\sR}-1} \left( 1-\lambda_1^j - \lambda_1(\lambda_1^{-j}-1) \right) \\
\frac{-1}{\delta_{\sR}-1} \left( 1-\lambda_1^j - \lambda_1(\lambda_1^{-j}-1) \right) &
\frac{\delta_{\sR}}{(\delta_{\sR}-1)^2} (1-\lambda_1^j) + \lambda_1(\lambda_1^{-j}-1)
\end{array} \right] \;.
\label{eq:sumMRLRj}
\end{equation}
Then (\ref{eq:Ps3j}) results by directly evaluating the determinant of
(\ref{eq:PS3j}) via the use of (\ref{eq:f}), (\ref{eq:paramRLRkLR}), (\ref{eq:MRLRk}) and (\ref{eq:sumMRLRj}).
(For simplicity the author achieved this using symbolic computations in {\sc matlab}\removableFootnote{
{\sc goSymF.m}.
}.)
\end{proof}

\begin{proof}[Proof of Theorem \ref{th:RLRkLR}]
By Lemma \ref{le:MS}, we have $0 < \det \left( M_{\cS[k]} \right) < 1$ and ${\rm trace} \left( M_{\cS[k]} \right) < 0$.
Thus for all $k \ge 1$, $\det \left( I - M_{\cS[k]} \right) \ne 0$, therefore by Lemma \ref{le:existence}, $\cS[k]$-cycles are unique.
Furthermore, we immediately see that the inequalities (\ref{eq:stabConditionSN}) and (\ref{eq:stabConditionNS}) hold for all $k \ge 1$.
The inequality (\ref{eq:stabConditionPD}) also holds for all $k \ge 1$ because by Lemma \ref{le:MS} we have
\begin{equation}
\det \left( M_{\cS[k]} \right) + {\rm trace} \left( M_{\cS[k]} \right) + 1 =
\frac{ \left( \delta_{\sR} + 1 \right)
\left( \left( \delta_{\sR}^2 + 1 \right)^{k+1} - \delta_{\sR}^{k+1} \right)}
{\delta_{\sR} \left( \delta_{\sR}^2 + 1 \right)^{k+1}} \;,
\nonumber
\end{equation}
which is positive.
Thus by Lemma \ref{le:stability} the $\cS[k]$-cycles are asymptotically stable
(assuming $\det \left( P_{\cS[k]^{(i)}} \right) \ne 0$ for all $i$, which is demonstrated below).

By Lemma \ref{le:admissibility}, for admissibility we require
\begin{equation}
\begin{split}
\det \left( P_{\cS[k]^{(3j)}} \right) &> 0 \;, {\rm ~for~} j = 0,\ldots,k-1 \;, \\
\det \left( P_{\cS[k]^{(3j+1)}} \right) &< 0 \;, {\rm ~for~} j = 0,\ldots,k-1 \;, \\
\det \left( P_{\cS[k]^{(3j+2)}} \right) &> 0 \;, {\rm ~for~} j = 0,\ldots,k-1 \;, \\
\det \left( P_{\cS[k]^{(3k)}} \right) &< 0 \;, \\
\det \left( P_{\cS[k]^{(3k+1)}} \right) &> 0 \;.
\end{split}
\label{eq:detPSi2}
\end{equation}
From (\ref{eq:Ps3j}) we can see that, as a function of $j$,
$\det \left( P_{\cS[k]^{(3j)}} \right)$ has a single turning point (at $j \approx \frac{k}{2}$)
that corresponds to a maximum.
Therefore, given $\delta_{\sR}$ and $k$,
over the range $j = 0,\ldots,k-1$, $\det \left( P_{\cS[k]^{(3j)}} \right)$ achieves its minimum at either $j = 0$ or $j = k-1$.
By (\ref{eq:Ps3j}), for $j = 0$:
\begin{align}
\det \left( P_{\cS[k]} \right) &=
\frac{\delta_{\sR}^4 + \delta_{\sR}^3 + \delta_{\sR} + 1 -
\left( \delta_{\sR}^4 + \delta_{\sR}^2 + \delta_{\sR} + 1 \right) \lambda_1^k}
{\left( \delta_{\sR}^2 - \delta_{\sR} + 1 \right) \left( \delta_{\sR}^2 + 1 \right)} \nonumber \\
&\ge \frac{\delta_{\sR}^2 \left( \delta_{\sR} - 1 \right)}
{\left( \delta_{\sR}^2 - \delta_{\sR} + 1 \right) \left( \delta_{\sR}^2 + 1 \right)} > 0 \;,
\nonumber
\end{align}
where we have substituted $k=0$ to produce the inequality.
Similarly for $j = k-1$:
\begin{align}
\det \left( P_{\cS[k]^{(3(k-1))}} \right) &=
\frac{\delta_{\sR}^4 + \delta_{\sR}^2 \left( \delta_{\sR} + 1 \right) \lambda_1^k
- \left( \delta_{\sR}^3 - 1 \right) \lambda^{k-1}}
{\left( \delta_{\sR}^2 - \delta_{\sR} + 1 \right) \left( \delta_{\sR}^2 + 1 \right)} \nonumber \\
&\ge \frac{\delta_{\sR}^4 - \delta_{\sR}^3 + 1}
{\left( \delta_{\sR}^2 - \delta_{\sR} + 1 \right) \left( \delta_{\sR}^2 + 1 \right)} > 0 \;.
\nonumber
\end{align}
Therefore $\det \left( P_{\cS[k]^{(3j)}} \right)$ is positive for all $\delta_{\sR} > 1$, $k \ge 1$ and $j = 0,\ldots,k-1$.
The remaining inequalities in (\ref{eq:detPSi2}) may be verified in the same fashion;
these calculations are omitted for brevity.
We then conclude that, for each $k \ge 1$, the unique $\cS[k]$-cycle is admissible and asymptotically stable.

Computations for $\cS'[k]$-cycles are analogous.
The key formulas are
\begin{align}
\det \left( M_{\cSp[k]} \right) &= \delta_{\sR}^2 \;, \nonumber \\ 
{\rm trace} \left( M_{\cSp[k]} \right) &= 
-\delta_{\sR} \lambda_1^k + \left( \delta_{\sR}^2 + \delta_{\sR} + 1 \right) \lambda_1^{-k} \;, \nonumber 
\end{align}
and
\begin{align}
\det \left( P_{\cSp[k]^{(3j)}} \right) &=
-\frac{1}{\delta_{\sR}^2-\delta_{\sR}+1} \bigg(
\delta_{\sR}^2 \left( \delta_{\sR}^2+\delta_{\sR}+1 \right) \left( \lambda_1^{-k}-\lambda_1^{-j}+\lambda_1^{k-j} \right) -
\delta_{\sR}^2 \left( \delta_{\sR}^2+1 \right) \nonumber \\
&\quad-\left( \delta_{\sR}^3-1 \right) \lambda_1^j \left( \lambda_1^{-k}-1 \right) -
\delta_{\sR}^3 \lambda_1^k \bigg) \;, \nonumber \\
\det \left( P_{\cSp[k]^{(3j+1)}} \right) &=
\frac{\delta_{\sR}}{\delta_{\sR}^2-\delta_{\sR}+1} \bigg(
\left( \delta_{\sR}^2+\delta_{\sR}+1 \right) \left( \delta_{\sR}^2+1 \right)
\left( \lambda_1^{-k}-\lambda_1^{-j}+\lambda_1^{k-j} \right) -
\left( \delta_{\sR}^2+1 \right)^2 \nonumber \\
&\quad-\delta_{\sR} \left( \delta_{\sR}^2+\delta_{\sR}+1 \right) \lambda_1^j \left( \lambda_1^{-k}-1 \right) -
\delta_{\sR} \left( \delta_{\sR}^2+1 \right) \lambda_1^k \bigg) \;, \nonumber \\
\det \left( P_{\cSp[k]^{(3j+2)}} \right) &=
-\frac{1}{\delta_{\sR} \left( \delta_{\sR}^2-\delta_{\sR}+1 \right)} \bigg(
\left( \delta_{\sR}^2+\delta_{\sR}+1 \right) \lambda_1^{-k} -
\frac{\delta_{\sR} \left( \delta_{\sR}^2+\delta_{\sR}+1 \right)}{\delta_{\sR}^2+1} \lambda_1^j \left( \lambda_1^{-k}-1 \right)
\nonumber \\
&\quad-\left( \delta_{\sR}^2+1 \right) \left( \delta_{\sR}^3-1 \right) \lambda_1^{-j} \left( \lambda_1^k-1 \right) -
\left( \delta_{\sR}^2+1 \right) - \delta_{\sR} \lambda_1^k \bigg) \;, \nonumber \\
\det \left( P_{\cSp[k]^{(3k)}} \right) &=
-\frac{1-\lambda_1^k}{\delta_{\sR}^2-\delta_{\sR}+1} \;, \nonumber \\
\det \left( P_{\cSp[k]^{(3k+1)}} \right) &=
-\frac{\delta_{\sR}^3 \left( 1-\lambda_1^k \right)}{\delta_{\sR}^2-\delta_{\sR}+1} \;, \nonumber 
\end{align}
where $\lambda_1$ is given by (\ref{eq:lambda1}).
From these formulas, uniqueness and admissibility of $\cS'[k]$-cycles for $k \ge 1$
follows in the same fashion as for $\cS[k]$-cycles.
\end{proof}

\section{Conclusions}
\label{sec:conc}

In this paper it is shown for the first time that the two-dimensional border-collision normal form (\ref{eq:f})
may have infinitely many coexisting attractors.
Theorem \ref{th:RLRkLR} states that with $\mu = 1$, $\delta_{\sR} > 1$ and (\ref{eq:paramRLRkLR}),
(\ref{eq:f}) has an attracting periodic solution with symbol sequence $\cS[k] = \left( \sR \sL \sR \right)^k \sL \sR$, for all $k \ge 1$.
Theorem \ref{th:RLRkLR} was proved by explicitly verifying all admissibility and stability conditions of the $\cS[k]$-cycles.

Fig.~\ref{fig:infFa} shows a plot of the $\cS[k]$-cycles with $\delta_{\sR} = \frac{3}{2}$.
As $k$ increases the $\cS[k]$-cycles approach an orbit that is homoclinic to an $\sR \sL \sR$-cycle.
Furthermore, the branches of the stable and unstable manifolds of the $\sR \sL \sR$-cycle that intersect are coincident.
In sections \ref{sec:necessary} and \ref{sec:further} it was shown that in general such coincidence is to be expected.
Given $\cX$ and $\cY$, if the $\cX$-cycle is of saddle-type and
(\ref{eq:f}) has infinitely many admissible, stable $\cX^k \cY$-cycles,
then, with some additional assumptions, (\ref{eq:f}) must display several important features.
Parts (\ref{it:gY}) and (\ref{it:lambda2}) of Theorem \ref{th:codim3} give three consequences
that imply that such coexistence is at least a codimension-three phenomenon.
In view of Theorem \ref{th:RLRkLR}, we conclude that the scenario is generically codimension-three.
By part (\ref{it:HCorbit}) of Theorem \ref{th:codim3}, the stable and unstable manifolds of the $\cX$-cycle intersect.
By Theorem \ref{th:coincident}, this intersection is non-transversal.

The results of \S\ref{sec:further} included the assumption that $(\cX \cY)_i \ne (\cY \cX)_i$ for $i=0$ and only one other index, call it $\alpha$.
It was shown that we expect the unstable manifold of the $\cX$-cycle to intersect the switching manifold at two points,
and for one of these points to map to the other under $\alpha$ iterations of (\ref{eq:f}).
By Theorem \ref{th:saddles}, $\cX^k \cY^{\overline{0}}$-cycles
limit to the homoclinic orbit of the $\cX$-cycle that includes these two points of intersection, as $k \to \infty$.

In \S\ref{sec:finding}, parameter values for which (\ref{eq:f}) exhibits infinite coexistence
were identified for three different combinations of $\cX$ and $\cY$.
These three examples each satisfy all the assumptions given in sections \ref{sec:necessary} and \ref{sec:further},
and the consequences listed above may be verified directly for these examples.

There are many avenues for future work.
It remains to remove some of the assumptions made in sections \ref{sec:necessary} and \ref{sec:further}, if possible,
and identify other mechanisms, if any exist, by which (\ref{eq:f}) may have infinitely many attractors.
As parameters are varied from a point at which there exist infinitely many attractors,
we would like to determine the rate at which the number of coexisting attractors decreases.
We would also like to understand exactly for which combinations of $\cX$ and $\cY$
(\ref{eq:f}) can exhibit infinitely many admissible, attracting $\cX^k \cY$-cycles.

Perhaps the most important problem that stems from this work is a generalization to the $N$-dimensional border-collision normal form.
In more than two dimensions calculations of periodic solutions can be performed in the same manner,
but stable and unstable manifolds of saddle-type periodic solutions may have a dimension greater than one, which presents more possibilities and difficulties.
Also, as noted in \cite{GlJe12,GlKo12}, it is not known how many attractors may be born simultaneously in grazing-sliding bifurcations.
The return map for grazing-sliding may be put in the border-collision normal form,
but it remains to demonstrate that parameter values that give rise to infinitely many coexisting attractors
are viable for grazing-sliding, and study the influence of higher-order terms.

\end{document}